\documentclass[a4paper,11pt]{article}
\usepackage{hyperref}
\hypersetup{colorlinks, citecolor=blue, filecolor=black, linkcolor=blue, urlcolor=blue}

\usepackage{amsmath}
\usepackage{amssymb}
\usepackage{theorem}
\usepackage{mathrsfs}
\usepackage[usenames,dvipsnames]{pstricks}
\usepackage{euscript}
\usepackage{graphicx}
\usepackage{charter}
\usepackage{ulem}
\newcommand{\email}[1]{\href{mailto:#1}{\nolinkurl{#1}}}
\usepackage[english]{babel}
\usepackage[utf8x]{inputenc}
\usepackage[T1]{fontenc}

\newtheorem{theorem}{Theorem}[section]
\newtheorem{lemma}[theorem]{Lemma}
\newtheorem{corollary}[theorem]{Corollary}
\newtheorem{proposition}[theorem]{Proposition}
\theoremstyle{plain}{\theorembodyfont{\rmfamily}%
}
\theoremstyle{plain}{\theorembodyfont{\rmfamily}%
}
\theoremstyle{plain}{\theorembodyfont{\rmfamily}%
}
\theoremstyle{plain}{\theorembodyfont{\rmfamily}%
}
\theoremstyle{plain}{\theorembodyfont{\rmfamily}%
\newtheorem{remark}[theorem]{Remark}}
\theoremstyle{plain}{\theorembodyfont{\rmfamily}%
\newtheorem{definition}[theorem]{Definition}}
\theoremstyle{plain}{\theorembodyfont{\rmfamily}%

\usepackage{bm}
\usepackage{subcaption}
\usepackage{float}
\usepackage{cancel}
\usepackage{mathtools}
\usepackage{enumitem}

\usepackage{appendix}
\usepackage{array}
\usepackage{footnote}
\usepackage{booktabs}

\usepackage{autonum}

\definecolor{blue}{HTML}{0000CD}

\definecolor{orange}{HTML}{FF8C00}

\DeclareMathOperator*{\argmin}{argmin}
\DeclareMathOperator*{\lsup}{limsup}
\newcommand{\xsol}{x^\dagger} 
\newcommand{\xtrue}{\bar x} 
\newcommand{\ytrue}{\bar y} 
\newcommand{\ynoisy}{\hat{y}}
\newcommand{\DDD}{\textbf{(3D)}~} 
\newcommand{\IDDD}{\textbf{(I3D)}~} 

\newcommand{\RA}{R_A^*}

\numberwithin{equation}{section}

\input{mystyle.sty}

\usepackage{algorithmic}

\newtheorem{example}[theorem]{Example}

\title{\sffamily\huge Accelerated iterative regularization \\ via dual diagonal descent}

\author{Luca Calatroni$^1$,\; Guillaume Garrigos$^2$,\; 
Lorenzo Rosasco$^3$, and Silvia Villa$^4$\\[5mm]
\small
\small $\!^1$Universit\'e C\^ote d'Azur, CNRS, Inria, I3S, France\\
\small \ttfamily{calatroni@i3s.unice.fr}
\\[5mm]
\small $\!^2$ LPSM, Universit\'e de Paris, Sorbonne Universit\'e, CNRS, Paris,  France\\
\small\ttfamily{garrigos@lpsm.paris}
\\[5mm]
\small
\small$\!^3$MaLGa, DIBRIS, Universit\`a di Genova, Via Dodecaneso 35, 16146 Genova, Italy, \\
\small CBMM, and  Massachussetts Institute of Technology, and  Istituto Italiano di Tecnologia\\
\small \ttfamily{lorenzo.rosasco@unige.it}
\\[5mm]
\small
\small $\!^4$MaLGa, DIMA, Universit\`a di Genova, Via Dodecaneso 35, 16146 Genova, Italy\\
\small \ttfamily{silvia.villa@unige.it}
}

\date{~}

\begin{document}

\maketitle

\begin{abstract}
We propose and analyze an accelerated iterative dual diagonal descent algorithm for the solution of linear inverse problems with general regularization and data-fit functions. In particular,  we develop an inertial approach of which we  analyze both convergence and stability. Using tools from inexact proximal calculus, we prove early stopping results with optimal convergence rates for  additive data-fit terms as well as  more general cases, such as the Kullback-Leibler divergence,  for which different type of proximal point approximations hold.
\end{abstract}

{\bfseries Keywords}. 
Iterative regularization, duality, acceleration,  forward-backward splitting, diagonal methods, stability and convergence analysis.

{\bfseries AMS Classification:}
90C25, 49N45, 49N15, 68U10,90C06.

\section{Introduction}  \label{sec:introduction}

We are interested in solving the linear inverse problem:
\begin{equation}   \label{eq:inverse_problem}
\text{find }\quad \xtrue \in \Xx \quad \text{s.t.}\quad A \xtrue = \ytrue,
\end{equation}
where $A: \mathcal{X}\to\mathcal{Y}$ is a bounded linear operator between two Hilbert spaces $\mathcal{X}$ and $\Yy$,
and $ \ytrue \in \mathcal{Y}$ can be seen as  a given  measurement of some unknown $\xtrue \in \Xx$ we want to recover.
In general, the inverse problem \eqref{eq:inverse_problem} is ill-posed as its solution (if it exists) may lack some fundamental properties like uniqueness or stability. 
A standard approach   is assuming the desired $\xtrue$ to be well-approximated by
 \cite{EngHanNeu96,BenBur18}:
\begin{equation}\label{eq:primal}
   \tag*{\textbf{($\bm{P_0(\ytrue)}$)}} \text{find\ }
      \xsol \in  \argmin  \left\{ R(x) \text{ s.t. } x\in\underset{x' \in \Xx}{\argmin}~ \ell(Ax';  \ytrue) \right\}, 
\end{equation} 
called the primal problem in the following.  Here,  $R$ is a  regularization function enforcing some \emph{a-priori} knowledge on  the desired solution $\xtrue$, while $\ell: \Yy^2\to\R\cup \{+\infty\}$ is a data-fit function. In practice, the  data  is  subject to noise  due to, e.g., possible transmission and/or acquisition problems. As a consequence, we only have access to an inexact version $\ynoisy$ of $\ytrue$. 
Using $\ynoisy$ in  \ref{eq:primal} does no longer provide a suitable solution and a 
standard approach  is Tikhonov regularization:
\begin{equation}\label{eq:primal penalized}
    \tag*{\textbf{($\bm{P_\lambda(\ynoisy)}$)}}
    \text{find\ } \hat{x}_\lambda \in  \argmin_{x\in \mathcal{X}}~ \left\{ p_\lambda(x):= R(x) + \frac{1}{\lambda}\ell(Ax;\ynoisy) \right\}.
\end{equation} 
Intuitively, the (so called) regularization parameter $\lambda>0$ balances the trust in the data $\ynoisy$ with the regularization effect enforced by $R$.  Finding a good approximate solution based on noisy data requires two  steps. First, problem  \eqref{eq:primal penalized} is solved for various choices of $\lambda$ by means of a suitable optimization algorithm.
Second, the computed solutions are compared using  some  validation criterion (e.g. discrepancy principles \cite{EngHanNeu96}, SURE \cite{Ste81,DelVaiFadPey14}, cross-validation \cite{SteChr08}) and  an optimal parameter $\lambda^*$ is computed along with the corresponding  solution $\hat{x}_{\lambda^*}$.
The above procedure  is in general very costly computationally and so called iterative regularization methods, that we study in this paper,  can provide accurate  and more efficient alternatives \cite{EngHanNeu96,BakKok05,KalNeuSch08}.

\paragraph{Iterative regularization}

Iterative regularization approaches find an approximation of $\xsol$  by running an iterative algorithm  and `stopping' it when close  to $\xsol$, \cite{BacBur09,BurOsh13,BurResHe07,BotHei12}.
The number of  iterations plays  the role of the regularization parameter, controlling at the same time accuracy and computations. Roughly speaking, an iterative regularization method can be described as follows: 
\begin{enumerate}
	\item \label{item:req1} Given any data $y \in \Yy$, the algorithm generates a sequence $(x_k(y))_{k\in\N}$  using   $R$, $\ell$ and $A$. 
	\item \label{item:req2} Given the true data $\ytrue$, the sequence $(x_k(\bar{y}))_\kin$ solves the hierarchical optimization problem \ref{eq:primal}, i.e.   converges  as $k\to\infty$ to the solution $x^\dagger$ of \ref{eq:primal}.
	\item \label{item:req3}  For a given  noise level $\delta >0$ and  noisy data $\hat{y}$ such that $\|\bar{y}-\hat{y}\|\leq \delta$, there is a \emph{stopping time} $k(\delta) \in \N$ such that 
	\begin{equation}\label{D:regularisation method abstract}
	  \Vert \hat{x}_{k(\delta)}(\hat y) -  x^\dagger \Vert = O(\delta^\alpha), \text{ for some } \alpha >0.  
	\end{equation}	
\end{enumerate}
The quantity $O(\delta^\alpha)$ is often called the \textit{rate} of the considered  regularization method,  and the exponent $\alpha$ quantifies its efficiency.

\paragraph{Previous results}
 For quadratic data-fit terms $\ell$ and square-norm regularization terms $R$, both Tikhonov and iterative regularization (such as the Landweber algorithm)  have been shown to be  \emph{optimal}, in the sense that their reconstruction error \eqref{D:regularisation method abstract} has optimal rate $O(\delta^{\frac{1}{2}})$ \cite{EngHanNeu96}.  
 Optimal results with possibly fewer iterations are also known to be possible considering accelerated approaches \cite{EngHanNeu96,Neubauer2016}. For quadratic data-fit, and general strongly convex regularizers  an iterative regularization procedure combined with a Morozov-type discrepancy principle was also shown to be optimal in \cite{BurResHe07}, and accelerated approaches based on  a dual accelerated gradient descent was shown to be optimal with  less iterations in \cite{MatRosVilVu17}. Iterative regularization methods have been studied also for convex regularizers in \cite{BurOsh13} where estimates  in terms of Bregman distance were proved (see, e.g., \cite{BurResHe07,BacBur09}  for Tikhonov-type approaches), but no explicit rates in the form \eqref{D:regularisation method abstract} were shown. More general iterative algorithms defined in Banach spaces have been studied  in \cite{Kaltenbacher_2009,Kaltenbacher_2013,Estatico2013} for linear and non-linear inverse problems and in \cite{BotHei12} for $L^1$ and Total Variation (TV) regularization. 
For results on iterative regularization for data-fit terms other than squared norm, we mention  \cite{BenningBurger2011} for results in the framework of Bregman distances and \cite{Garrigos2017} where a dual diagonal descent \DDD algorithm is considered. To the best of our knowledge, accelerated iterative regularization approaches have not been studied in this general setting.

\paragraph{Contribution and organization of the paper} In this paper,  we study a novel accelerated iterative regularization algorithm for  strongly convex regularization terms and general data-fits and prove its optimality. As for the quadratic case \cite{Neubauer2016,MatRosVilVu17}, we show that  acceleration can lead to optimal results in much fewer iterations. 
Our approach, dubbed \IDDD, extends the  \DDD iterative algorithm studied in \cite{Garrigos2017} introducing an inertial term  yielding acceleration.
From an optimization perspective, the rationale behind these results is that inertial dynamics are able to exploit information in previous iterates to converge faster to an optimal solution. However, as pointed out in \cite{Devolder2014}, differently from basic schemes, inertial methods  suffer from errors accumulation that need to be controlled along the iterations and balanced with the improvements observed in the convergence speed.\\
The  paper is organized as follows.
In Section \ref{sec:post-introduction} we introduce the main notations  and assumptions. In Section \ref{sec:dynamicall} we introduce and analyze a diagonal inertial dynamic in continuous time.  In Section \ref{sec:algorithm}, we derive \IDDD as a discretization of the continuous dynamic, and we study its convergence and stability properties.  In Section \ref{sec:stability} the results are illustrated for a number of specific data-fit terms including Kullback-Leibler divergence data-fit.

\section{Main assumptions and background on  diagonal methods}\label{sec:post-introduction}
We begin fixing the  notation. Let $\Hh$ be a Hilbert space with scalar product  $\langle\cdot,\cdot\rangle$ and associated norm $\|\cdot\|$. Given $y\in\Hh$, and $\varrho\in\mathbb{R}_+$ and $\mathbb{B}(y,\varrho)$ is the open ball of center $y$ and radius $\varrho$.
we denote by $\Gamma_0(\Hh)$ the set of proper, convex and lower semi-continuous functions from $\Hh$ to $]-\infty,+\infty]$.
We say that  $f \in \Gamma_0(\Hh)$ is $\sigma$-{\em strongly convex}  if 
$f-{\sigma}\Vert \cdot \Vert^2/2 \in \Gamma_0(\Hh)$, with  $\sigma \in\left]0,+\infty\right[$.
We recall that the {\em subdifferential} of  $f \in \Gamma_0(\Hh)$  is the multi-valued operator 
$\partial f : \Hh \rightarrow 2^{\Hh}$ defined by setting
\begin{equation}  \label{def:subdiff}
(\forall x\in\Hh)\quad \partial f(x) := \left\{u \in \Hh: f(x') - f(x) - \langle u , x' - x \rangle \geq 0 \text{ for all } x' \in \Hh \right\}.
\end{equation}
Note that if $f$ is Gateaux differentiable at $x \in \Hh$, then $\partial f(x)=\{\nabla f(x) \}$, see, e.g.~, \cite[Proposition 17.31 i)]{BauschkeCombettes2017}.
We also recall that for all  $x\in \Hh$ and $\tau>0$ the proximity operator $\prox_{\tau f} : \Hh \rightarrow \Hh$  of $f\in\Gamma_0(\Hh)$ with parameter $\tau$ is defined by:
\begin{equation}\label{def:proximal operator}
    \prox_{\tau f}(x) =
    (I + \partial f)^{-1}(x) = 
    \argmin\limits_{x' \in \Hh} \left\{ f(x') + \frac{1}{2\tau}\Vert x' - x \Vert^2 \right\}.
\end{equation}
Given $f \in \Gamma_0(\Hh)$, we will denote by $f^* : \Hh \rightarrow [-\infty,+\infty]$   the {\em Fenchel conjugate} of $f$
\[
(\forall u\in\Hh)\quad f^*(u):=  \sup_{x\in \Hh}  \left\{\langle u,x\rangle_{\Hh} -f(x) \right\}.
\]
The Fenchel conjugate of $f$ belongs to $\Gamma_0(\Hh)$, and if $f$ is $\sigma$-strongly convex then $f^*$ is differentiable at any point, with a $\sigma^{-1}$-Lipschitz continuous gradient, see, e.g. \cite[Theorem 18.15]{BauschkeCombettes2017}. 
Furthermore, the following property holds, see \cite[Theorem 16.23]{BauschkeCombettes2017}: 
\[
(\forall (x,u)\in \Hh^2)\quad u \in \partial f(x) \Leftrightarrow x \in \partial f^*(u).
\] 
Finally, given two real sequences $(a_k)_{k\geq 1}$ and $(b_k)_{k\geq 1}$, we will write $a_k = O(b_k)$ whenever there exists a positive constant $M>0$ such that $a_k \leq M b_k$ for all $k \geq 1$. We will use the more precise notation $a_k = \Theta(b_k)$ if both conditions $a_k = O(b_k)$ and $b_k = O(a_k)$ hold.
Note also that we will use the same notation  $\|\cdot\|$ and $\langle\cdot,\cdot\rangle$ for the norm and the scalar product in every Hilbert space we consider. 

\subsection{Main assumptions}
We make   the following assumptions on the data-fit $\ell$ and the regularizer $R$: 
\begin{enumerate}[label=\textbf{$(L_\arabic*)$}]
	\item \label{assumption:data-fit convex and coercive}
	For all $y \in \Yy$, $\ell_y:=\ell(\cdot,y)\in \Gamma_0(\Yy) $ and it is coercive.
	\item \label{assumption:data-fit not flat} 
	For all $(y_1,y_2)\in \Yy^2, ~\ell(y_1,y_2)=0 \iff y_1=y_2$.
	\item \label{assumption:data-fit conditionned locally}
	 there exists
     $q\in[1,+\infty[$, $\varrho \in ]0,+\infty]$, $ \gamma \in ]0,+\infty[$ such that
	$$
	\forall y \in \mathbb{B}( \ytrue,\varrho),\quad \frac{\gamma}{q}\Vert y-\ytrue\Vert^q \leq \ell(y,\ytrue). 
	$$ 
\end{enumerate}
\begin{enumerate}[label=\textbf{$(R_\arabic*)$}]	
	\item \label{H:R strongly convex} $R$ is $\sigma$-strongly convex, with $\sigma \in \left]0,+\infty\right[$,  
	\item \label{assumption:AR2}   $\partial R(\xsol) \cap \im A^* \neq \emptyset$.	
\end{enumerate}
If assumption \ref{assumption:data-fit conditionned locally} is satisfied, we say that $\ell_{\ytrue}$ is \emph{locally
	$q$-conditioned} at $\bar{y}$. If \ref{assumption:data-fit conditionned locally} is satisfied with $\rho=+\infty$, then we say that $\ell_{\ytrue}$  is $q$-conditioned at $\bar{y}$.
These assumptions  cover a wide range of inverse problems, as discussed next.

\begin{definition}\label{D:additive datafit}
 A data-fit is additive if there exists $\mathcal{N} \in \Gamma_0(\Yy)$ such that
\begin{equation}
 (\forall (y_1,y_2)\in\Yy^2) \quad  \ell(y_1,y_2) = \mathcal{N}(y_1-y_2).
\end{equation}
\end{definition}

\begin{example}[Data-fit functions]\label{R:data-fit function list}
For $\Yy=\R^d$, the additive data-fit functions defined by the functions $\mathcal{N}$ below trivially satisfy \ref{assumption:data-fit convex and coercive}-\ref{assumption:data-fit not flat}. In addition, $\ell_{\bar{y}}$ satisfies \ref{assumption:data-fit conditionned locally} if and only if $\mathcal{N}$ is locally conditioned at $0$, and this is the case, as we show below:
\begin{enumerate}[label=$\bullet$]
    \item $\mathcal{N}(y) = \frac{1}{2}\Vert y \Vert^2$ is $2$-conditioned at $0$, with $\gamma =1$.
    \item $\mathcal{N}(y) = \frac{1}{q}\Vert y \Vert_q^q$, for $q\geq 1$, is $q$-conditioned at $0$ with  $\gamma = d^{r}$, where $r:=\min (\frac{1}{q} - \frac{1}{2}, 0)$. Note that this includes the case of the $\ell^1$-norm.
    \item the weighted sum \cite{HinLan13} $\mathcal{N}(y) = \alpha\Vert y \Vert_1 + \frac{1}{2}\Vert y \Vert_2^2$, for $\alpha >0$, is $1$-conditioned at $0$, with $\gamma=\alpha$.
    \item the Huber data-fit function \cite{CalDeSch17} $\mathcal{N}(y) = \sum\limits_{i=1}^d h_\nu(y^{i})$, where $h_\nu \colon \R\to\R_{+}$ is the Huber smoothing function, defined for $\nu > 0$ as
\begin{equation}
\label{e:huber}
(\forall t\in \R)\quad h_\nu(t):=
\begin{cases}
\frac{1}{2 \nu }t^2  & \text{ if } \vert t \vert \leq \nu \\
 \vert t \vert -\frac{\nu}{2}& \text{ otherwise.}
\end{cases}
\end{equation}
The Huber data-fit is locally $2$-conditioned at $0$. It is enough to choose $\varrho \in \left]0,+\infty\right[$ with 
$\gamma = \min\{{1}/{\nu},{2 \varrho - \nu}/{\varrho^2}\}$. 
\item the exact penalization $\Nn(y)=0$ if $y=0$, $\Nn(y)=+\infty$ otherwise, is $1$-conditioned with $\gamma =1$.
\end{enumerate}

We also mention a non-additive data-fit function used in several applications:
\begin{enumerate}[label=$\bullet$]
    \item the Kullback-Leibler divergence, defined as:
    \begin{equation}  \label{definition:KL}
    \ell(y_2,y_1)=\mathrm{KL}(y_1,y_2):=\sum_{i=1}^d \mathrm{kl}(y_1^{i},y_2^{i}),
    \end{equation}
    where
    \[
    (\forall (t_1,t_2)\in \R^2)\quad
    \mathrm{kl}(t_1,t_2):=\begin{cases} \displaystyle t_{1}\log\frac{t_{1}}{t_2}-t_{1}+t_2 &\text{if }  (t_1,t_2) \in\left ]0,+\infty\right[^{2}, \\[2ex] 
    +\infty &\text{otherwise.}\end{cases}
    \]
    The Kullback-Leibler divergence is locally $2$-conditioned at $\bar{y}$ for every $\varrho \in ]0, +\infty[$, with $\gamma = \frac{2}{\varrho c^2} + \frac{2}{\varrho^2 c}\ln{\frac{c}{\varrho + c}}$, with $c = d \Vert  \bar{y} \Vert_\infty$ (see Lemma \ref{P:conditioning of KL}).
\end{enumerate}
\end{example}

\begin{example}[Regularizers]\label{E:regularizers list}
A regularizer widely used in signal/image processing as a sparsifying prior is the $\ell^1$-norm of the coefficients 
with respect to an orthonormal basis, or a more general dictionary. 
Another popular choice in imaging is the total variation semi-norm 
\cite{RudOshFat92}, due to its ability to preserve edges, together with its generalizations \cite{BreKunPoc10,ChaLio97}.
For some specific tasks in computer vision and machine learning, there is a need for structured sparsity. 
This kind of prior can be enforced with the use of group sparsity inducing norms \cite{YuaLin06,Bac09}. While not being strongly convex, 
these regularizers can be included in our framework
by adding a strongly convex term $\frac{\sigma}{2}\Vert \cdot \Vert^2$ where  $\sigma$ is small positive parameter,  in the flavor of the elastic net regularization \cite{ZouHas05}. 
\end{example}

\subsection{Iterative methods based on continuous and discrete dynamics}

It is useful to  review some  approaches to solve \eqref{eq:inverse_problem}, the hierarchical problem \ref{eq:primal} and the Tikhonov-regularized
problem \ref{eq:primal penalized}. In  particular, we  focus on approaches based on duality and/or combined with diagonal dynamics.

\paragraph{Mirror descent  approaches}
A class of methods to solve \eqref{eq:inverse_problem}  consider the   problem 
\begin{equation}   \label{eq:primal_char}
    \text{find }  \xsol \in  \argmin_{x \in \Xx}  \left\{ R(x) + \delta_{\bar{y}}(Ax) \right\},
\end{equation}
where the constraint \eqref{eq:inverse_problem} is encoded by the indicator function $\delta_{\bar{y}}$. Using Fenchel-Rockafeller duality,  
we derive the corresponding dual problem:
\begin{equation}   \label{eq:Dual}
\text{find } u^\dagger \quad\text{ s.t. }\quad u^\dagger \in \argmin_{u\in \mathcal{Y}}\  \left\{ d_0(u) := R^*(-A^*u) + \langle \bar{y}, u \rangle \right\}
    \tag*{\textbf{($\bm{D_0}$)}}.
\end{equation}
Since $R^*$ is smooth (see \ref{L:functional properties:L at 0} in Lemma \ref{L:functional properties}), a gradient  method can  be  used to solve \ref{eq:Dual}, see \cite{MatRosVilVu17}. This coincides, up to a change of variables, with mirror descent approaches \cite{BeckTeboulleMirror2003} and linearized Bregman iterations \cite{BurResHe07,BacBur09},  where $R$ plays the role of the mirror function.  How to extend this  approach to solve \ref{eq:primal} is not clear. 

\paragraph{Primal diagonal dynamics}

A classical approach to solve hierarchical problems like \ref{eq:primal} is the \textit{diagonal principle}, based on the fact that when $\ynoisy = \bar y$ and $\lambda\to 0$, problem \ref{eq:primal penalized} converges towards \ref{eq:primal} in an appropriate sense \cite[Theorem 2.6]{Att96}.  In this view, diagonal approaches have  been considered as non-autonomous  dynamics solving \ref{eq:primal penalized} with a parameter $\lambda$ monotonically decreasing to zero.
The simplest example of a continuous diagonal dynamic is the diagonal steepest descent differential inclusion with initial $t_0>0$ defined by
\begin{equation}
    \label{ODE:dynamic primal diagonal first order}
    x(t_0)=x_0, \quad \lambda(t) \searrow 0, \quad 
    \dot{x}(t) + \partial p_{\lambda(t)}(x(t)) \ni 0,
   \tag*{\textbf{($\bm{PD_\lambda)}$}}
\end{equation}
where $p_{\lambda(t)}(x(t))$ is defined in \ref{eq:primal penalized}. This dynamic is studied in \cite{AttCom96,AttCza10,AttCabCza18} where
convergence of $x(t)$ to $x^\dagger$ was guaranteed provided that $\lambda(t)\to 0$ \textit{fast enough}, namely $\lambda \in L^{1/(q-1)}([t_0,+\infty))$, where $q\in[1,+\infty)$ is the  exponent  in \ref{assumption:data-fit conditionned locally}, see
\cite[Corollary 3.3, Remark 4.4]{AttCabCza18}.
Discrete counterparts of \ref{ODE:dynamic primal diagonal first order} have also been studied  \cite{BahLem94,AttCzaPey11,CzaNouPey16}.  They can be seen as a
variant of the Forward-Backward algorithm applied to solve problem \ref{eq:primal penalized}, where the penalization parameter tends to zero along the 
iterations. A  main drawback of this type of algorithms is that they are expansive for non-smooth data-fit terms, since they require to compute the proximal
operator of $\ell_{\ytrue} \circ A$. A possible way to overcome this issue  is applying Fenchel-Rockafellar duality to \ref{eq:primal penalized}.  The main  advantage of considering the dual problem \ref{eq:DualReg} is that the linear operator appears there only in composition with the smooth function $R^*$. Then it is possible to apply an explicit gradient step to $R^* \circ (-A^*)$, while the non-smooth data-fit term can be cheaply treated via its proximal operator.

\paragraph{Dual diagonal dynamics}

The dual problem of \ref{eq:primal penalized} is
\begin{equation}   \label{eq:DualReg}
\text{find } u_\lambda \in \argmin_{u\in \mathcal{Y}}\ \left\{d_\lambda(u) := R^*(-A^*u) + \frac{1}{\lambda}\ell^*(\lambda u;\ynoisy)\right\}.   \tag*{\textbf{($\bm{D_\lambda}$)}}
\end{equation} 
Solutions of \ref{eq:DualReg} are related  to those of \ref{eq:primal penalized} via the formula $x_\lambda = \nabla R^*(-A^* u_\lambda)$, which holds thanks to the strong convexity of $R$. A natural question is whether the diagonal principle can be applied on to the dual problem \ref{eq:DualReg} as well. The corresponding dual diagonal continuous dynamics read
\begin{equation}  \label{ODE:dynamic dual diagonal first order}
u(t_0)=u_0, \quad \lambda(t) \searrow 0, \quad 
\begin{cases}
x(t) = \nabla R^*(-A^*u(t)), \\
\dot{u}(t) 
+\partial d_{\lambda(t)}(u(t)) 
\ni0.
\end{cases},
\tag*{\textbf{($\bm{DD_\lambda}$)}}
\end{equation}
where, similarly as before, provided that $\lambda \in L^{1/(q-1)}([t_0,+\infty))$, the trajectory $x(t)$ is guaranteed to converge to $x^\dagger$.
The  discrete counterpart of \ref{ODE:dynamic dual diagonal first order} is called Dual Diagonal Descent algorithm \DDD, and its convergence and stability properties are studied in \cite{Garrigos2017}.  For $\ynoisy\in\Yy$ such that $\Vert \hat y - \bar y \Vert \leq \delta$ and  additive data-fit functions, the authors showed that stopping the algorithm at  $k_\delta =\Theta(k^{-2/3})$ guarantees that \eqref{D:regularisation method abstract} holds with $\alpha = 1/3$. However, that this rate is not optimal  \cite{EngHanNeu96}.  As we show in the following, in this paper our approach recover the optimal rate, while providing the benefits of accelerated approaches, namely an earlier stopping time.

\section{Continuous inertial dual diagonal dynamic}  \label{sec:dynamicall}
First-order inertial algorithms are popular in optimization due to their faster convergence on smooth and non-smooth convex problems, see e.g. \cite{Nes04,BecTeb09}. In several papers  continuous inertial dynamics have been studied considering appropriate Lyapunov functions \cite{Candes2016,KriBayBar15,ApiAujDos17}.
As  already noted, their regularization properties are also known for quadratic data-fit terms \cite{Neubauer2016,MatRosVilVu17}. 

Next, we propose an inertial approach for general data-fit terms, introducing a variant of the dynamic in  \ref{ODE:dynamic dual diagonal first order}. 
For a given $\alpha>0$ and initial $t_0>0$, let 
\begin{equation}  \label{ODE:dynamic}
(u(t_0),\dot{u}(t_0))=(u_0,\dot{u}_0), \quad \lambda(t) \searrow 0, \quad
\begin{cases}
x(t) = \nabla R^*(-A^* u(t)), \\
\ddot{u}(t) 
+ \dfrac{\alpha}{t}\dot{u}(t)
+\partial d_{\lambda(t)}(u(t)) 
\ni 0.
\end{cases}   
\tag*{\textbf{($\bm{IDD_\lambda}$)}}
\end{equation}
The asymptotic behavior of the trajectories of this  inertial differential inclusion will be analyzed next,  while its discrete counterpart will be studied in the rest of the paper.
We first add one remark. 
\begin{remark}
The idea of coupling  inertia with Tikhonov regularization is not new.
In \cite{AttChbRia18} an inertial variant of the primal dynamic \ref{ODE:dynamic primal diagonal first order}  is proposed for  $R=\Vert \cdot \Vert^2/2$. The corresponding inertial primal diagonal approach is:
\begin{equation}  \label{ODE:dynamic primal diagonal second order}
(x(t_0),\dot{x}(t_0))=(x_0,\dot{x}_0), \quad \lambda(t)\searrow 0, \quad 
\ddot{x}(t) + \frac{\alpha}{t} \dot{x}(t) + \lambda(t) \partial p_{\lambda(t)}(x(t)) \ni 0.
   \tag*{\textbf{($\bm{IPD_\lambda}$)}}
\end{equation}
Under a suitable decay assumption on $\lambda(\cdot)$ the authors  guarantee  fast convergence  and  regularization  \cite[Section 6]{AttChbRia18}.
Compared to \ref{ODE:dynamic primal diagonal second order}, in our dual formulation \ref{ODE:dynamic} we  take advantage of a different scaling between the data-fit and the regularizer. Indeed, in \ref{ODE:dynamic} the data-fit is  multiplied by $\lambda(t)^{-1} \to + \infty$, while in \ref{ODE:dynamic primal diagonal second order} the regularizer is multiplied by $\lambda(t) \to 0$. For first-order systems this difference is essentially cosmetic, the two approaches  being equivalent, for  an appropriate change of variables \cite{AttCza10}.
However, for second-order systems these two scalings describe different dynamics \cite[Section 4]{AttCza17}. This difference can be understood  looking at  the limits (in the $\Gamma$-convergence sense) of the corresponding parametrized functions, indeed,
\begin{equation}
    \text{ if } \lambda\searrow 0, \quad p_\lambda \to p_0 := R + \delta_{\argmin \ell_y \circ A} \quad \text{ and } \quad \lambda p_\lambda \to \delta_{\dom R} + \ell_y \circ A.
\end{equation}
\end{remark}

\subsection{Convergence of the continuous inertial dual diagonal dynamic}

Next, we study the convergence properties of the trajectories of \ref{ODE:dynamic}, assuming their existence
to simplify the analysis.
We  remark that if $d_\lambda$ is assumed to be gradient-Lipschitz continuous, global existence and uniqueness results of a classical $C^2([t_0,+\infty),\R_+)$ solution to \ref{ODE:dynamic} hold from the Cauchy-Lipschitz theorem.
However, this assumption requires  the data-fidelity function $\ell_{\ytrue}$ to be strongly convex (see \cite[Theorem 18.15]{BauschkeCombettes2017}), 
which is in general not the case for most of the data-fit terms, see Example \ref{R:data-fit function list}.
We refer to \cite{CabPao07,ApiAujDos17} for further details.
In the following Theorem, we show that the inertial term  \ref{ODE:dynamic} ensures that the dual function values $d_{\lambda(t)}(u(t))$ tend to $\inf d_0$ at a $O(t^{-2})$ rate as expected for inertial methods. Further, switching from the dual to the  primal problem by means of $x(t) = \nabla R^*(-A^* u(t))$, we   prove the convergence of  $x(\cdot)$  to $x^\dagger$.  To prove  these results, assumptions on the decay of  $\lambda(\cdot)$ are needed, as it is usual for dynamics such as those in \ref{ODE:dynamic primal diagonal first order} and \ref{ODE:dynamic primal diagonal second order}.
{ For instance, in \cite{AttCza10,AttCabCza18}, it is shown  that the trajectories of \ref{ODE:dynamic primal diagonal first order} converge under the condition $\lambda\in L^{\frac{1}{q-1}}([t_0,+\infty))$ \cite{AttCabCza18}.}
Similarly, we consider the following assumption:
\begin{enumerate}[label=\textbf{$(\Lambda)$}]
\item \label{requirement:lambda1:cont} $\lambda: \left[0,+\infty\right[ \to \left]0,+\infty\right[$ is a non-increasing differentiable function such that $\lim_{t\to\infty} \lambda(t) = 0$. 
Moreover, if $q$ defined in assumption \ref{assumption:data-fit conditionned locally} is strictly greater than $1$,  the quantity $\Lambda_c = \int_{t_0}^{+\infty} t \lambda^{\frac{1}{q-1}}(t)~dt$ is finite.
\end{enumerate}
\begin{remark}  \label{remark:summability:cont}
A sufficient condition ensuring the validity of \ref{requirement:lambda1:cont} is that $\lambda(\cdot)\in L^{\frac{1}{2(q-1)}}([t_0,+\infty))$, see Lemma \ref{lemma:integrability} in the Appendix. 
\end{remark}
We are now ready to state the main convergence result for continuous dynamics. Note that Lemma~\ref{L:functional properties}\ref{L:functional properties:argmin d_0 nonempty} ensures that the set of solutions of problem \ref{eq:Dual} is nonempty. To prove  fast convergence results of the dual function values, we follow the approach considered in \cite{Attouch2015,Candes2016,ApiAujDos17} and define a suitable Lyapunov-type function.
\begin{theorem}\label{T:CV continuous dynamic rates}
Let the assumptions \ref{assumption:data-fit convex and coercive}-\ref{assumption:data-fit conditionned locally}, \ref{H:R strongly convex}-\ref{assumption:AR2}, \ref{requirement:lambda1:cont} hold true. Let $u^\dagger\in\argmin d_0$ and assume that  $\lambda(t_0) \Vert u^\dagger  \Vert \leq \gamma \varrho^{q-1}/q$.
Let $\alpha \geq 3$ and let the pair $(x(\cdot),u(\cdot))$ be a solution to \ref{ODE:dynamic} in the following sense:
\begin{enumerate}[label=$\bullet$]
    \item $u \in \mathcal{C}^1([t_0, +\infty[, \Yy)$, and $x = \nabla R^* \circ (-A^*) \circ u$,
    \item for every ~$T > t_0$, $\dot u$ and $d_{\lambda(\cdot)} \circ u$ are absolutely continuous on $[t_0,T]$,
    \item for a.e. $t \in [t_0, + \infty[$, $-\ddot{u}(t) - \frac{\alpha}{t} \dot u(t) \in \partial d_{\lambda(t)}(u(t))$.
\end{enumerate}
Then, there exists an explicit  $C\in\left]0,+\infty\right[$ such that
\[
\forall t > t_0 \quad d_{\lambda(t)}(u(t)) - \inf d_0 \leq \frac{C}{t^2}\quad\text{ and}\quad \|x(t) -x^\dagger\| \leq \frac{\sqrt{2C}}{\sqrt{\sigma} t}.
\]
\end{theorem}

\begin{proof}
Define the following energy:
\begin{equation}  \label{def:Lyapunov}
(\forall t\geq t_0)\quad\mathcal{E}(t):= t^2 \left( d_{\lambda(t)} (u(t)) - \inf  d_0\right) +\frac{1}{2} \|(\alpha-1)( u(t) - u^\dagger) + t\dot{u}(t)\|^2.
\end{equation}
From now on we will use the following shorthand notation: 
\begin{equation} \label{eq:composite}
   \RA: = R^* \circ (-A^*),\qquad \ell^*_{\bar{y}}(\cdot): = \ell(\cdot,\bar{y})^*
\end{equation}
so that the composite dual function $d_\lambda$ can be written as $d_{\lambda}(u) =  \RA(u) + \lambda^{-1} \ell^*_{\bar{y}}(\lambda u)$, for every $u\in\mathcal{Y}$.
Since $\partial d_{\lambda(t)}(u(t)) = \nabla \RA(u(t)) + \partial \ell^*_{\bar y }(\lambda(t) u(t))$ \cite[Proposition 16.6 and Corollary 16.53]{BauschkeCombettes2017}, the notion of solution introduced entails that there exists some $\eta : [t_0,+\infty) \rightarrow \Yy$ such that
\[
    \text{ for a.e. $t>t_0$},\quad  \ddot{u}(t) + \frac{\alpha}{t}\dot{u}(t) + \nabla \RA(u(t)) +\eta(t) =0 \ \text{ and } \ \eta(t) \in \partial \ell^*_{\bar y }(\lambda(t) u(t)).
\]
We divide the proof in two steps.

\paragraph{Step 1. Fast convergence rates} 
The function $\mathcal{E}$ is  differentiable a.e. on $[t_0,+\infty[$ since it is absolutely continuous.
We thus compute its derivative and obtain: 
\begin{align}
\dot{\mathcal{E}}(t) & = 2t \Big(  d_{\lambda(t)} (u(t)) - \inf  d_0 \Big) + \frac{t^2\dot{\lambda}(t)}{\lambda^2(t)}\Big( \langle \eta(t),\lambda(t)u(t) \rangle - \ell^*_{\bar{y}}(\lambda(t)u(t)) \Big) \notag \\
& + t^2 \langle \dot{u}(t), \ddot{u}(t) + \frac{\alpha}{t}\dot{u}(t) + \nabla \RA(u(t))+\eta(t) \rangle + t(\alpha-1) \langle u(t)- u^\dagger, \frac{\alpha}{t}\dot{u}(t) + \ddot{u}(t) \rangle. \notag
\end{align}
The second term in the expression above is non-positive because $\lambda$ is differentiable and decreasing and, moreover, by convexity of $\ell_{\bar{y}}(\cdot)$ together with Lemma \ref{L:functional properties}\ref{L:functional properties:L at 0}, there holds
$$
  \ell^*_{\bar{y}}(\lambda(t)u(t))-\langle  \eta(t),\lambda(t)u(t) \rangle \leq  \ell^*_{\bar{y}}(0)=0.
$$ 
Furthermore, the third term is equal to zero a.e. since $u(\cdot)$ is a solution of \ref{ODE:dynamic} by assumption. 
We  thus  deduce that for a.e. $t>t_0$
\begin{equation}\label{proof:ccd1}
\dot{\mathcal{E}}(t)  \leq 2t \Big(  d_{\lambda(t)} (u(t)) - \inf  d_0 \Big) +t(\alpha-1)\langle u^\dagger - u(t), -\ddot{u}(t) -\frac{\alpha}{t}\dot{u}(t)  \rangle.
\end{equation}
Using that $-\ddot{u}(t) -\frac{\alpha}{t}\dot{u}(t) \in \partial d_{\lambda(t)}(u(t))$ and from the convexity of $d_{\lambda(t)}(\cdot)$ we have:
\begin{equation}
\text{for a.e. } t>t_0 \quad \langle u^\dagger-u(t), -\ddot{u}(t) -\frac{\alpha}{t}\dot{u}(t) \rangle \leq d_{\lambda(t)}(u^\dagger)- d_{\lambda(t)}(u(t)).   
\end{equation}
We now add and subtract $ \inf  d_0 = d_0(u^\dagger)$ and define  $r_{\lambda(t)}(u^\dagger):= d_{\lambda(t)}(u^\dagger) - \inf  d_0$. 
We get:
\begin{equation}
\langle u^\dagger-u(t), -\ddot{u}(t) -\frac{\alpha}{t}\dot{u}(t) \rangle \leq  r_{\lambda(t)}(u^\dagger) + \inf  d_0 -d_{\lambda(t)}(u(t)) .    
\end{equation}
Applying this inequality to \eqref{proof:ccd1}, since $\alpha\geq 3$ and $d_{\lambda(t)} (u(t)) - \inf  d_0 \geq 0$ {(see Proposition \ref{L:functional properties}.\ref{L:functional properties:decreasing function sequence})}, we get:
\begin{equation}\label{proof:ccd2}
   \dot{\mathcal{E}}(t) \leq  t(3-\alpha)\Big(  d_{\lambda(t)} (u(t)) - \inf  d_0 \Big)+ t(\alpha-1)r_{\lambda(t)}(u^\dagger) \leq t(\alpha-1)r_{\lambda(t)}(u^\dagger) . 
\end{equation}
To bound the right hand side, we now apply Lemma \ref{L:functional properties}\ref{L:functional properties:diagonal rates via conditioning} and deduce that, 
since $\lambda(t)\leq\lambda(t_0)$:
$$
\dot{\mathcal{E}}(t) \leq c(\alpha-1)  t\lambda(t)^{\frac{1}{q-1}},
$$
where the constant $c$ is defined as:
\begin{equation}\label{proof:ccd3}
    c :=
    \begin{cases}
    0 \quad &\quad \text{ if }q=1,\\
    (1-(1/q))\gamma^{-1/(q-1)}\Vert u^\dagger \Vert^{q/(q-1)} \quad &\quad \text{ if } q >1,
    \end{cases}
\end{equation}
and it is finite in both cases.
Since the above inequality holds for a.e. $t > t_0$, assumption \ref{requirement:lambda1:cont} yields that for a.e. $t > t_0$:
$$
\mathcal{E}(t) = \mathcal{E}(t_0) + \int_{t_0}^{t} \dot{\mathcal{E}}(t) \leq \mathcal{E}(t_0) + c(\alpha-1) \Lambda_c.
$$
Defining $C:= \mathcal{E}(t_0)+ c(\alpha-1) \Lambda_c$, we derive
\begin{equation}   \label{eq:rates}
 d_{\lambda(t)}(u(t)) - \inf   d_0 \leq \frac{C}{t^2}.
\end{equation}

\paragraph{Step 2. Convergence rate for the primal iterates} 

From \eqref{eq:rates} in combination to Lemma \ref{L:functional properties}\ref{L:functional properties:from dual value to primal iterates}, we get
\begin{eqnarray*}
\frac{\sigma}{2}\| x(t) - x^\dagger\|^2 & \leq &  d_0(u(t)) - \inf   d_0 = (   d_0(u(t)) -    d_{\lambda(t)}(u(t)) ) + (   d_{\lambda(t)}(u(t))  - \inf  d_0 ) \\
& \leq & (  d_0(u(t)) -    d_{\lambda(t)}(u(t)) ) + \frac{C}{t^2}.
\end{eqnarray*}
The monotonicity property of Lemma \ref{L:functional properties}\ref{L:functional properties:decreasing function sequence} implies that the first term on the right hand side above is non-positive, whence we get
\begin{equation}  \label{eq:trajectories}
\|x(t) - x^\dagger\| \leq \frac{\sqrt{2C}}{\sqrt{\sigma} t}.
\end{equation}
\end{proof}

\section{Inertial Dual Diagonal Descent \IDDD Algorithm}  \label{sec:algorithm}

In this section,  we study the convergence properties of the discrete analogue of \ref{ODE:dynamic}, deriving an accelerated version of the \DDD algorithm  in \cite{Garrigos2017}.

\subsection{From the continuous dynamic to the discrete algorithm}

We follow a standard approach  of computiong the time-discretization of the  continuous dynamical system \cite{Alvarez2000,AttouchPeypRed2014,Candes2016,Attouch2015}.
Recalling the notation  in \eqref{eq:composite}, we note  
that \ref{ODE:dynamic} can be equivalently written as
\begin{equation} \label{eq:ODE_composite1}
\begin{cases}
x(t) = \nabla R^*(-A^* u(t)), \\
\ddot{u}(t) + \frac{\alpha}{t}\dot{u}(t) + \partial \ell^*_{\bar{y}}(\lambda(t)u(t)) + \nabla  \RA(u(t)) \ni 0.
\end{cases}
\end{equation}
We  discretize \eqref{eq:ODE_composite1} \emph{explicitly}  with respect to the smooth component $\nabla  \RA$ and  \emph{semi-implicitly} with respect to the non-smooth one $\partial \ell^*_{\bar{y}}$. In other words, we discretize implicitly the trajectories, while leaving explicit the dependence on the discretized values $\lambda_k$. For $k\geq 0$, a fixed time step-size $h>0$ and for some time discretization points $t_k=kh$, we set $u_k:=u(t_k)$,  $\lambda_k:=\lambda(t_k)$ and derive the  finite difference scheme:
\[
\begin{cases}
    x_k = \nabla R^*(-A^*u_k), \\
    \frac{1}{h^2}(u_{k+1} -2u_k+u_{k-1})+\frac{\alpha}{kh^2}(u_k-u_{k-1}) + \partial \ell^*_{\bar{y}}(\lambda_{k}u_{k+1}) + \nabla  \RA(w_k) \ni 0,
\end{cases}
\]
where $w_k$ is a linear combination of $u_k$ and $u_{k-1}$, made clear  in the following. 
After straightforward calculations, we rewrite the system above as
\begin{equation}   \label{eq:ID1_FD}
\begin{cases}
    x_k = \nabla R^*(-A^* u_k), \\
    u_{k+1} + h^2\partial \ell^*_{\bar{y}}(\lambda_{k}u_{k+1})  \ni u_k + \left(1-\frac{\alpha}{k}\right) ( u_k - u_{k-1}) - h^2\nabla \RA(w_k).
\end{cases}
\end{equation}
Hence, by setting $\alpha_k = 1-\alpha/k$, $\tau:=h^2$ and $w_k :=  u_k +\alpha_k ( u_k - u_{k-1})$,
we get
\[
\begin{cases}
 w_k & = u_k + \alpha_k(u_k-u_{k-1}), \\
 u_{k+1} &
 = \left(I + \frac{\tau}{\lambda_{k}}\partial \ell^*_{\bar{y}}(\lambda_{k}\cdot)\right)^{-1}\left(w_k - \tau\nabla \RA(w_k)\right),\\
    x_{k+1} & = \nabla R^*(-A^* u_{k+1}). 
\end{cases}
\]
Note that the proximal operator of the map $\ell^*_{\bar{y}}(\lambda_{k}\cdot)$ with parameter $\tau/\lambda_k$ appears, in combination with an explicit gradient step for $R_A^*$. We can thus introduce the  Inertial Dual Diagonal Descent \IDDD algorithm
\begin{equation}\label{alg:fw_accel}
u_0 = u_1 \in \Yy, \quad \text{ compute for } k \geq 1 \quad 
\begin{cases}
 w_k & = u_k + \alpha_k(u_k-u_{k-1}), \\
 u_{k+1} &= \mathrm{prox}_{\frac{\tau}{\lambda_{k}} \ell^*_{\bar{y}}(\lambda_k\cdot)}\left(w_k - \tau\nabla \RA(w_k)\right),\\
    x_{k+1} & = \nabla R^*(-A^* u_{k+1}). 
\end{cases}
\tag*{\textbf{($\bm{I3D}$)}}
\end{equation}
This algorithm depends on three parameters: the stepsize $\tau >0$, the relaxation parameters $(\lambda_k)_k$ and the friction parameters $(\alpha_k)_k$.
The stepsize will be chosen depending on the value of the Lipschitz constant of $\nabla \RA$.
For the relaxation parameters, we will consider a discrete analogue of the assumption \ref{requirement:lambda1:cont} formulated in the continuous setting.
For the friction parameters $\alpha_k$, we will allow more general values not necessarily corresponding to $\alpha_k = 1 -\alpha/k$ as  above. We gather  the requirements on the   parameters in the following assumptions:
\begin{enumerate}[label=\textbf{$(P_\arabic*)$}]
	\item \label{assumption:parameter:stepsize}
	$\tau \in \left(0,\frac{\sigma^2}{\Vert A \Vert^2}\right]$, where $\sigma > 0$ is defined in assumption \ref{H:R strongly convex}.
	\item \label{assumption:parameter:friction}
	$\alpha_k$ is non-negative and for every $k\geq 1$, $t_k:=1+\sum_{i=k}^{+\infty}\prod_{j=k}^i \alpha_j$ is finite,  with $t_k = \Theta(k)$.
	\item \label{assumption:parameter:relaxation}
	$(\lambda_k)$ is a strictly positive non-increasing sequence such that $\lim_{k\to\infty} \lambda_k = 0$. 
	Moreover, by defining
	\begin{equation}   \label{def:biglambda}
	\Lambda:=
	\begin{cases}
	\sum_{k\geq 1} t_{k+1} \lambda_k^{1/(q-1)}
	& \text{ if } q>1, \\
	0
	& \text{ if } q=1, 
	\end{cases}
	\end{equation}
	we have that $\Lambda< + \infty$.
	\item \label{assumption:parameter:technical}
	For some $u^\dagger \in \argmin d_0$, we have  $\lambda_0 \Vert u^\dagger \Vert \leq \gamma \varrho^{q-1} /q$.
\end{enumerate}

\begin{remark}[On  assumption \ref{assumption:parameter:relaxation}]\label{R:on eq:cond:summability}
As commented in Remark \ref{remark:summability:cont}, one can  check that a sufficient condition for \ref{assumption:parameter:relaxation} to hold is that $\lambda\in \ell^{\frac{1}{2(q-1)}}(\N)$. 
In particular, if we consider a sequence  verifying $\lambda_k = O\left(k^{-\theta}\right)$ for some $\theta > 0$,
it is easy to verify that \ref{assumption:parameter:relaxation} holds as long as $\theta > 2(q-1)$.
For  $q=1$,  (for instance if $\ell(y_1,y_2) = \Vert y_1-y_2\Vert_1$), no summability condition is required.
Roughly speaking, the assumption $\lambda\in \ell^{\frac{1}{2(q-1)}}(\N)$ means that  $\lambda\in \ell^{\infty}(\N)$, which 
is already implied by $\lim_{k\to\infty} \lambda_k = 0$.
\end{remark}

\begin{remark}[On  assumption \ref{assumption:parameter:technical}]\label{R:on eq:cond:technical}
For many choices of data-fits, $\varrho = +\infty$ (see Example \ref{R:data-fit function list}), in which case the assumption is automatically satisfied.
Also, note  that in assumption \ref{assumption:parameter:relaxation}, we require $\lambda_k$ to tend to zero. This means that $\lambda_K \Vert u^\dagger \Vert \leq \gamma \varrho^{q-1} /q$ for some $K \in \N$. In this case, up to a time rescaling $k \leftarrow k+K$  our estimates hold true.
\end{remark}
 Following \cite{AttouchTikAccel}, we require the sequence of friction parameters $(\alpha_k)$ to satisfy  \ref{assumption:parameter:friction},  a particular summability property  guaranteeing a technical condition crucial in the following. We summarize such a requirement and the resulting condition in the following lemma.

\begin{lemma}[ {{\cite[Lemma 2.1]{AttouchTikAccel} }}]   \label{lemma:sequence}
Assume that $(\alpha_k)$ is non-negative and satisfies
\begin{equation}  \label{eq:choice:alpha}
\sum_{i=k}^{+\infty}\prod_{j=k}^i \alpha_j<+\infty,\quad\text{for every }k\geq 1.
\end{equation}
Then, the sequence defined by
\begin{equation} \label{eq:choice:tk}
t_k:=1+\sum_{i=k}^{+\infty}\prod_{j=k}^i \alpha_j
\end{equation}
is well-defined, and satisfies for every $k\geq 1$ the following properties:
\begin{equation}   \label{eq:properties_sequences_alphak_tk}
1+\alpha_k t_{k+1} = t_k, \qquad t^2_{k+1}-t^2_k\leq t_{k+1}.
\end{equation}
\end{lemma}

\begin{remark}[Classical choices of $\alpha_k$ and $t_k$] \label{rmk:classical_choices_tk}
Definitions \eqref{eq:choice:alpha} and \eqref{eq:choice:tk} above accommodate standard choices of sequences $(\alpha_k)$ and $(t_k)$. For example, in his seminal work  Nesterov \cite{Nesterov1983} considered
\begin{equation}   \label{eq:Nesterov_alpha}
\alpha_k=\frac{t_k-1}{t_{k+1}} \qquad\text{and}\qquad t_{k+1}=\frac{\sqrt{1+4t_k^2}+1}{2}, \qquad  t_1=1,
\end{equation}
which can be shown to verify the two conditions \eqref{eq:choice:alpha} and \eqref{eq:choice:tk}, as well as $k/2 \leq t_k \leq k$.
For a given $\alpha>1$, the two asymptotically equivalent choices 
$$
\alpha_k=1-\frac{\alpha}{k},\quad t_{k+1} = \frac{k}{\alpha-1},\quad\text{and}\quad \alpha_k=\frac{k-1}{k+\alpha-1},\quad t_{k+1} = \frac{k+\alpha-1}{\alpha-1}
$$
have been recently considered in \cite{ChambolleDossal2015, Apidopoulos2017, AttouchNesterovsubcritical2017} and  can be shown to  satisfy \ref{assumption:parameter:friction}.
For $\alpha=3$, these sequences are  asymptotically equivalent to  Nesterov sequences \eqref{eq:Nesterov_alpha}.
\end{remark}

\begin{remark}[Splitting of the loss]  \label{rmk:inf_conv}
In \cite{Garrigos2017} the decomposition $\ell_{\ytrue}= \phi_{\ytrue}~ \Box~ \psi_{\ytrue}$ was considered, where $\Box$ is the infimal convolution  and $\psi_{\ytrue}$ is the possible strongly convex component of $\ell_{\ytrue}$. In this case, the dual function $\ell_{\ytrue}^*(\cdot)$ can be  expressed as
$
\ell_{\ytrue}^* = \psi_{\ytrue}^*+  \phi_{\ytrue}^*,
$
where $\phi_{\ytrue}^*$ is in general non-smooth, while $\phi_{\ytrue}^*$ has Lipschitz gradient and can therefore be incorporated with the smooth term $ \RA$ in the dual function $d_\lambda$.
For several data discrepancies, however, $\psi_{\ytrue}= \delta_{\left\{0\right\}}$ (see \cite[Section 4.3]{Garrigos2017}). To  simplify the presentation, we do not consider this decomposition. 
\end{remark}

\subsection{Fast convergence of the algorithm}

We now prove the discrete analogue of Theorem \ref{T:CV continuous dynamic rates} for  \ref{alg:fw_accel}. We follow the approach  considered in \cite{AttouchNesterov,Candes2016,Attouch2015,AttouchTikAccel}.
\begin{theorem}[Fast convergence]\label{T:discrete convergence rates noiseless}
Let the assumptions \ref{assumption:data-fit convex and coercive}-\ref{assumption:data-fit conditionned locally}, \ref{H:R strongly convex}-\ref{assumption:AR2}, \ref{assumption:parameter:stepsize}-\ref{assumption:parameter:technical} hold true.
Let $(x_k)$ and $(u_k)$ be the sequences generated by algorithm \ref{alg:fw_accel}.
Then, there exists $C\in\left]0,+\infty\right[$ such that
\begin{equation}  \label{eq:rate}
 d_{{\lambda_k}}(u_k) - \inf  d_0  \leq \frac{C}{t_k^2} 
 \quad \text{ and } \quad 
 \| x_k - x^\dagger \| \leq \frac{\sqrt{2C}}{\sqrt{\sigma}t_k}.
\end{equation}
\end{theorem}
\begin{proof}
Let $u^\dagger\in \argmin d_0$ be the minimizer of $d_0$ defined in assumption \ref{assumption:parameter:technical}, and
define, for every $k\geq 1$,  the discrete Lyapunov energy function:
\begin{equation}  \label{eq:discrete energy}
\mathcal{E}(k):=t_k^2 \Big( d_{{\lambda_k}}(u_k)-\inf d_0\Big) + \frac{1}{2\tau}\|z_k-u^\dagger\|^2,
\end{equation}
where $z_k$ is defined as:
\begin{equation}  \label{eq:def_zk:noisefree}
z_k:=u_{k-1} +t_k(u_k-u_{k-1}).
\end{equation}
Our goal is to get an estimate on the decay of $\mathcal{E}$ along time.
More precisely, we will show that for every $k\geq 1$
\begin{equation}  \label{eq:condition}
\mathcal{E}(k+1) - \mathcal{E}(k)
\leq  t_{k+1} \Big(d_{{\lambda_k}}(u^\dagger) -  \inf d_0\Big),
\end{equation}
which can be seen as a discrete analogue of \eqref{proof:ccd2}, and from which the desired accelerated convergence rates will follow in a straightforward manner.

For simplicity, let us denote by $\ell_k$  the function defined by setting
\begin{equation}  \label{eq:shortand_notation_Dk}
(\forall u\in \mathcal{Y})\quad \ell_k(u)=\lambda_k^{-1}{\ell}_{\bar{y}}^*(\lambda_k u).
\end{equation} 
To prove \eqref{eq:condition}, we define for every $k\geq 1$ the operator $G_k:\mathcal{Y}\to \mathcal{Y}$ as
\begin{equation}  \label{eq:fw_operator}
G_k(z)  := \frac{1}{\tau}\Big(z-\text{prox}_{\tau \ell_k}(z-\tau \nabla  \RA(z))\Big)
\end{equation}
and notice that the proximal step of  \ref{alg:fw_accel} can be written in terms of $G_k$ as $u_{k+1}=w_k - \tau G_k(w_k)$. The descent lemma yields (see, e.g., \cite{AttouchTikAccel,ChambolleDossal2015})
\begin{equation}  \label{eq:descent}
    d_{\lambda_k}(w-\tau G_k(w)) \leq  d_{\lambda_k}(u) + \langle G_k(w), w-u\rangle - \frac{\tau}{2}\|G_k(w)\|^2, \quad \text{for all }w,u\in\mathcal{Y},
\end{equation}
Evaluating \eqref{eq:descent} for $u=u_{k}$ and $w=w_k$, we get
\begin{equation}  \label{descent:first}
d_{{\lambda_k}}(u_{k+1}) \leq  d_{{\lambda_k}}(u_{k}) + \langle G_k(w_k), w_k - u_{k}\rangle - \frac{\tau}{2}\|G_k(w_k)\|^2.
\end{equation}
Evaluating \eqref{eq:descent} for $u=u^\dagger$ and $w=w_k$ we derive
 \begin{equation}  \label{descent:second}
d_{{\lambda_k}}(u_{k+1}) \leq  d_{{\lambda_k}}(u^\dagger) + \langle G_k(w_k), w_k -u^\dagger\rangle - \frac{\tau}{2}\|G_k(w_k)\|^2.
\end{equation}
We now multiply \eqref{descent:first} by $t_{k+1}-1$ and we add it to \eqref{descent:second}, and we obtain
\begin{align}
t_{k+1} d_{{\lambda_k}}(u_{k+1}) 
&\leq  (t_{k+1}-1)d_{{\lambda_k}}(u_{k}) +  d_{{\lambda_k}}(u^\dagger)  \nonumber\\
\label{ineq:1}&+ \langle  G_k(w_k), (t_{k+1}-1)(w_k-u_{k}) + (w_{k} - u^\dagger)\rangle  -\frac{\tau}{2}t_{k+1}\|  G_k(w_k)\|^2 . 
\end{align}
As an immediate consequence of Lemma \ref{lemma:sequence}, we observe the following property:
\begin{align}  \label{eq:property_zk}
(t_{k+1}-1)(w_k-u_{k}) + w_k & = u_{k} + t_{k+1}(w_k -u_{k}) \notag\\
& = u_{k} + t_{k+1} \alpha_k(u_{k} -u_{k-1}) \\
& = u_{k-1} + (1+ t_{k+1} \alpha_k)(u_{k} -u_{k-1}) \notag\\
& = u_{k-1} +t_k(u_{k} -u_{k-1}) = z_k. \notag
\end{align}
Thanks to \eqref{eq:def_zk:noisefree}, the fact that $z_k-\tau t_{k+1}G_k(w_k)=z_{k+1}$ and the previous equality, we can reorder the terms in \eqref{ineq:1} and rewrite it as
\begin{align}
t_{k+1} ( d_{{\lambda_k}}(u_{k+1}) -  d_{{\lambda_k}}(u^\dagger)) 
& \leq  (t_{k+1}-1) (d_{{\lambda_k}}(u_{k}) -  d_{{\lambda_k}}(u^\dagger)) \notag\\
& + \frac{1}{2\tau t_{k+1}}\left(\|z_{k}-u^\dagger\|^2-\|z_{k+1}-u^\dagger\|^2\right). \notag 
\end{align}
We now multiply everything by $t_{k+1}$, re-arrange and get
\begin{align}  \label{descent:third}
& t^2_{k+1} ( d_{{\lambda_k}}(u_{k+1}) -  d_{{\lambda_k}}(u^\dagger))   + \frac{1}{2\tau}
\|z_{k+1}-u^\dagger\|^2   \\
&\leq  (t^2_{k+1}-t_{k+1}) (d_{{\lambda_k}}(u_{k}) -  d_{\lambda_k}(u^\dagger)) + \frac{1}{2\tau } \|z_{k}-u^\dagger\|^2.  \notag
\end{align}
Equivalently:
\begin{align}  \label{descent:third2}
& t^2_{k+1} \Big( d_{{\lambda_k}}(u_{k+1}) -  d_{{\lambda_k}}(u^\dagger) \Big)   + \frac{1}{2\tau}
\|z_{k+1}-u^\dagger\|^2 \nonumber  \\
&\leq  t^2_k \Big(d_{{\lambda_k}}(u_{k}) -  d_{{\lambda_k}}(u^\dagger)\Big) +   (t^2_{k+1}-t_{k+1} - t_{k}^2) \Big( d_{{\lambda_k}}(u_{k}) -  d_{{\lambda_k}}(u^\dagger) \Big)+ \frac{1}{2\tau } \|z_{k}-u^\dagger\|^2. \notag
\end{align}
To get the desired terms, we first use on the left-hand side  the monotonicity property of the function $ d_{{\lambda_k}}(\cdot)$ as a function of $k$ (see  Lemma  \ref{L:functional properties}\ref{L:functional properties:decreasing function sequence}) and then add and subtract in the parentheses  the term $\inf d_0$, thus getting:
\begin{align}  \label{descent:fifth}
&t^2_{k+1} \Big( d_{{\lambda_{k+1}}}(u_{k+1}) -   \inf d_0 \Big)   + \frac{1}{2\tau}
\|z_{k+1}-u^\dagger\|^2   \notag \\
&\leq  t^2_k \Big(d_{{\lambda_{k}}}(u_{k}) - \inf d_0 \Big) +   (t^2_{k+1}-t_{k+1} - t_{k}^2) \Big( d_{{\lambda_{k}}}(u_{k}) -  \inf d_0\Big) \\
& +t_{k+1} \Big(d_{{\lambda_{k}}}(u^\dagger) -  \inf d_0\Big)   +  \frac{1}{2\tau } \|z_{k}-u^\dagger\|^2. \notag
\end{align}
After rearranging and using the definition of $\mathcal{E}$ in \eqref{eq:discrete energy}, we deduce:
\begin{align}  \label{descent:fifth}
\mathcal{E}(k+1) +    (t^2_{k}+t_{k+1} - t_{k+1}^2) \Big( d_{{\lambda_{k}}}(u_{k}) -  \inf d_0\Big) \leq  \mathcal{E}(k) + t_{k+1} \Big(d_{{\lambda_{k}}}(u^\dagger) -  \inf d_0\Big). \nonumber
\end{align}

\noindent Thanks to \eqref{eq:properties_sequences_alphak_tk} and  Lemma~\ref{L:functional properties}\ref{L:functional properties:decreasing function sequence}, we can neglect the second term  in the left-hand side of the above inequality,  finally getting the desired  \eqref{eq:condition}.
Iterating \eqref{eq:condition}  recursively gives
\begin{align}  \label{eq:decay_complete}
\mathcal{E}(k) \leq \mathcal{E}(1) +\sum_{j=1}^{k-1}t_{j+1}\Big( d_{\lambda_{j}}(u^\dagger)-\inf d_0 \Big).
\end{align}
To  bound the sum appearing on the right hand side, we need to analyze the residuals $r_j:= d_{\lambda_{j}}(u^\dagger) -\inf d_0$. Similarly as for the estimation obtained in the continuous case, we can use for this purpose the property in Lemma \ref{L:functional properties}\ref{L:functional properties:diagonal rates via conditioning} and get that for some fixed constant $c$ independent on $j$ (defined analogously as in \eqref{proof:ccd3}), we have
$$
r_j\leq c \lambda_j^{\frac{1}{q-1}},\qquad\text{for every }j\geq 1.
$$
By assumption \ref{assumption:parameter:relaxation}, with $\Lambda$ as  in \eqref{def:biglambda}, we thus conclude that
$$
\sum_{j=1}^{k-1} t_{j+1} r_j \leq  c \sum_{j=1}^{k-1} t_{j+1} \lambda_j^{\frac{1}{q-1}}  \leq c \Lambda < +\infty
$$
This allows us to deduce from \eqref{eq:decay_complete} the convergence rate on the dual values in \eqref{eq:rate}, by taking $C = \Ee(1) + c \Lambda$.
Finally, the convergence rate on the primal iterates  in \eqref{eq:rate} follows from  Lemma \ref{L:functional properties}\ref{L:functional properties:from dual value to primal iterates}.
\end{proof}

\begin{remark}[Nesterov scheme as a special case]\label{R:optimality of rates}
Let $f$ be any differentiable function in $ \Gamma_0(\Xx)$ with a Lipschitz gradient.
Take $R=f^*$, $A=-I$, $\bar y=0$, and $\ell(y_1,y_2) = \delta_{0}(y_2 - y_1)$, so that assumptions \ref{assumption:data-fit convex and coercive}-\ref{assumption:data-fit conditionned locally} and \ref{H:R strongly convex}-\ref{assumption:AR2} are verified.
In that case, $d_0 = f$, and \ref{alg:fw_accel} reads
\[
u_0 = u_1 \in \Yy, \quad \text{ compute for } k \geq 1 \quad 
\begin{cases}
 w_k & = u_k + \alpha_k(u_k-u_{k-1}), \\
 u_{k+1} &= w_k - \tau\nabla f(w_k),\\
    x_{k+1} & = \nabla f(u_{k+1}),
\end{cases}
\]
which in the dual exactly performs Nesterov's method \cite{Nes04}.
From our rates and Lemma \ref{L:functional properties}\ref{L:functional properties:decreasing function sequence}, we deduce that $f(u_k) - \inf f = O(k^{-2})$. 
And, according to Nemirovski and Yudin optimality result \cite[Theorem 2.1.7]{Nes04}, these rates are optimal over the class of Lipschitz smooth convex functions.
\end{remark}

\begin{remark}[Different growth for $t_k$]\label{R:different choices t_k}
In assumption \ref{assumption:parameter:friction} we ask the sequence $(t_k)$ to satisfy $t_k = \Theta(k)$, but this is actually not used in the proof of Theorem \ref{T:discrete convergence rates noiseless}.
What is crucial  is  that $t_k <+\infty$, so that Lemma \ref{lemma:sequence} can be used.
Indeed, one might ask whether it is possible to require $t_k =\Theta( k^\beta)$, with $\beta > 1$ to  improve the rates in \eqref{eq:rates}.
It is a simple exercise to verify that this is not possible, since \eqref{eq:properties_sequences_alphak_tk} implies $t_k \leq t_1 k$, hence  $\beta \leq 1$ so that the best rates are achieved for $\beta=1$.
\end{remark}

\section{Stability properties in the presence of errors}  \label{sec:errors}

We  now study the iterative regularization properties of  \ref{alg:fw_accel} in the presence of noisy data, given by
\begin{equation}\label{alg:fw_accel1}
\hat{u}_0 = \hat u_1 \in \Yy, \quad \text{ compute for } k \geq 1 \quad 
\begin{cases}
 \hat w_k & = \hat u_k + \alpha_k(\hat u_k-\hat u_{k-1}), \\
 \hat u_{k+1} &= \mathrm{prox}_{\frac{\tau}{\lambda_{k}} \ell^*_{\hat {y}}(\lambda_k\cdot)}\left(\hat w_k - \tau\nabla \RA(\hat w_k)\right),\\
    \hat x_{k+1} & = \nabla R^*(-A^* \hat u_{k+1}). 
\end{cases}
\end{equation}
A first question is how much the dual and primal iterates $\hat{u}_k$ and $\hat{x}_k$ are affected by  noise in terms of both convergence and stability. 
We discuss these issues  showing that the noise  can be interpreted as an error in the calculation of the proximal step of the algorithm.  Before starting, we motivate our analysis with the following example.

\begin{example}\label{E:stab_err}
Assume  $\Yy = \R$ and $\hat y = \bar y + \delta$, for some $\bar y,~ \delta >0$.
The algorithm makes use of the datum only in the evaluation of the proximal operator $\prox_{\frac{\tau}{\lambda}\ell^*_{\hat y}(\lambda \cdot)}$. One way to measure the impact of the noise is to find an upper bound for $\vert \prox_{\frac{\tau}{\lambda}\ell^*_{\bar y}(\lambda \cdot)}(w) - \prox_{\frac{\tau}{\lambda}\ell^*_{\hat y}(\lambda \cdot)}(w) \vert$, for  $w \in \Yy$  (see \cite[Lemma 10]{Garrigos2017}. As two particular cases consider:
\begin{itemize}
	\item $\ell_y= \frac{1}{2}\vert \cdot -y \vert^2$. We have:
	\[
	\sup\limits_{\bar y \in \Yy}\sup\limits_{w \in \Yy} \vert \prox_{\frac{\tau}{\lambda}\ell^*_{\bar y}(\lambda \cdot)}(w) - \prox_{\frac{\tau}{\lambda}\ell^*_{\hat y}(\lambda \cdot)}(w) \vert = \frac{\tau \delta}{1+ \tau \lambda}.
	\]
	\item $\ell_y = \mathrm{kl}(y;\cdot)$. We have:
\[
\sup\limits_{\bar y \in \Yy}\sup\limits_{w \in \Yy} \vert \prox_{\frac{\tau}{\lambda}\ell^*_{\bar y}(\lambda \cdot)}(w) - \prox_{\frac{\tau}{\lambda}\ell^*_{\hat y}(\lambda \cdot)}(w) \vert =  \sqrt{\frac{\tau \delta }{\lambda}}.
\]
\end{itemize}
In the former example, the error assumed in the evaluation of $\bar y$ is propagated along the iterations with order $\delta$.
However, a different behavior is observed for the latter example. The square-root dependence on $\delta$ makes the estimate worse in a small noise regime, when $\delta\ll 1$. 
Further,  the diagonal convergence to zero of the sequence $(\lambda_k)$, assumed in \ref{assumption:parameter:relaxation}, makes the overall error growing fast along the iterations.
\end{example}
Example~\ref{E:stab_err} shows that not all losses behave the same. Our analysis needs to be flexible enough to take these differences into account, and  avoid  sub-optimal results via  a worst-case analysis. This is the purpose of this section, where we will see that additive losses (in the sense of Definition \ref{D:additive datafit}) behave essentially like $\frac{1}{2}\vert \cdot -y \vert^2$, while  Kullback-Leibler belongs to a class of less stable losses.

\subsection{$\varepsilon$-subdifferentials and inexact proximal calculus}  \label{sec:preliminaries}
In this section, we define perturbations in a precise way. We first recall standard definitions regarding the approximate subdifferential and proximal-type minimization problems.
\begin{definition}[$\varepsilon$-subdifferential \cite{ZalinescuBook2002}]
Let $\mathcal{H}$ be a Hilbert space,  $f\in \Gamma_0(\mathcal{H})$ and $\varepsilon\geq 0$. The $\varepsilon$-subdifferential of $f$ at $x\in\mathrm{dom}~ f$ is the set
\begin{equation}   \label{eq:eps_subdifferential}
\partial_\varepsilon f(x) = \left\{ u\in \mathcal{H} : f(x') \geq f(x) + \langle u, x'-x \rangle -\varepsilon,\text{ for all }x'\in \mathcal{H} \right\}.
\end{equation}
\end{definition}
Such a notion generalizes that of the  subdifferential recalled in \eqref{def:subdiff}.
In particular, if $\varepsilon \geq 0$, then  $\partial f(x) \subset \partial_{\varepsilon} f(x)$ for any $x\in\mathcal{H}$, and we have
\begin{equation}  \label{eps_subdif:charact}
0\in \partial_\varepsilon f(x)\quad \Longleftrightarrow\quad
x \in \text{argmin}^{\eps} \, f =\{ x' \in \Hh : f(x') \leq \inf f + \varepsilon \}.
\end{equation}
The following are useful characterizations of the proximal operator of  $f\in\Gamma_0(\mathcal{H})$ with parameter $\eta>0$,
\begin{equation}\label{eq:prox characterisations}
p=\prox_{\eta f}(x) \,
\Leftrightarrow \, 
\frac{x-p}{\eta} \in \partial f(p) \, 
\Leftrightarrow \,
p = \argmin_z~\left\{ f(z) + \frac{1}{2\eta} \Vert z - x \Vert^2 \right\}.
\end{equation}
Next, we introduce notions of approximation of proximal points that can be seen as relaxed conditions of the characterizations in \eqref{eq:prox characterisations} (for  details see \cite{SalzoVilla2012, AujolDossalStability2015}).

\begin{definition}[Approximation of proximal points]  \label{def:approximation}
Let $f\in \Gamma_0(\mathcal{H})$, $x \in \Hh$, $\eta >0$ and $p:=\mathrm{prox}_{\eta f} (x)$. 
We say that  $\hat{p}\in\mathcal{H}$ is: \vspace{0.1cm}\\
    $\bullet$ a \textbf{type $1$} approximation of $p$ with precision $\varepsilon_1$, and we write $\hat{p}\approx^{\varepsilon_1}_{\text{\tiny 1}} p$, if:
    \begin{align}
        & \exists e \in \Hh,~ \exists (\eps_1,\eps_2,\eps_3) \in [0,+\infty[^2,~ \Vert e \Vert \leq \varepsilon_3, \ \eps_2^2 + \eps_3^2 \leq \eps_1^2,~  \frac{x + e - \hat p}{\eta} \in \partial_{\frac{\eps_2^2}{2\eta}} f(\hat p). \notag
    \end{align}
    $\bullet$ a \textbf{type $2$} approximation of $p$ with precision $\varepsilon_2$, and we write $\hat{p}\approx^{\varepsilon_2}_{\text{\tiny 2}} p$, if
    \begin{equation}
         \exists \eps_2\in [0,+\infty[,~ \frac{x - \hat p}{\eta} \in \partial_{\frac{\eps_2^2}{2\eta}} f(\hat p).
    \end{equation}
    $\bullet$ a \textbf{type $3$} approximation of $p$ with precision $\varepsilon_3$, and we write $\hat{p}\approx^{\varepsilon_3}_{\text{\tiny 3}} p$, if
    \begin{equation}
        \exists e \in \Hh, \exists \eps_3\in [0,+\infty],~ \Vert e \Vert \leq \varepsilon_3,~ ~\frac{x + e - \hat p}{\eta} \in \partial f(\hat p).
    \end{equation}
\end{definition}
Type 3 approximation simply describes an additive error in the argument of the proximal map, i.e. $\hat p = \prox_{\eta f}(x+e)$.
We show in Section \ref{sec:distance_based} that this type of error arises naturally when additive-type data-fit functions are used.
Type 2 approximation considers an error in the subdifferential operator.
The definition of type 1 approximation can be seen as a combination of type 2 and 3 approximations,
and the following Lemma provides an easy characterization.
\begin{lemma}[\cite{SchmidtLeRouxBach2011,SalzoVilla2012}]\label{lemma:error equivalence and implications}
Let $f \in \Gamma_0(\mathcal{H})$, $x \in \Hh$, $\eta >0$.
Then:
    \begin{equation}
        \hat{p}\approx^{\varepsilon_1}_{\text{\tiny 1}} \mathrm{prox}_{\eta f} (x) \quad
		\Leftrightarrow \quad        
        \hat p \in \text{\rm argmin}_{\eps_1}~\left\{ f(\cdot) + \frac{1}{2\eta} \Vert \cdot - x \Vert^2 \right\}.
    \end{equation}
\end{lemma}
In summary, those three type of errors correspond to relaxations of the characterizations in  \eqref{eq:prox characterisations}.
We are ready to study the stability properties of the \IDDD algorithm.

\subsection{Stability estimates in the presence of errors}  \label{sec:error_estimates}

Using the notions introduced  in the previous section, we can quantify the error due to  replacing  $\bar{y}$ by $\hat{y}$.
In particular, recalling Definition \ref{def:approximation}, we assume that at each iteration the proximal step   with $\hat{y}$ is an $i$-type approximation of the proximal step  with $\bar{y}$, where $i \in \{1,2,3\}$.
\begin{enumerate}[label=${(E_i)}$]
    \item \label{assumption:approx_i} 
    For every $k \geq 1$, $\exists {\varepsilon_{i,k}} \in [0,+\infty[,\text{ s.t. }\forall w \in \Yy$,
    \[
    \prox_{\frac{\tau}{\lambda_k}\ell^*_{\hat y}(\lambda \cdot)}(w) \approx^{\varepsilon_{i,k}}_{\text{\tiny $i$}} 
    \prox_{\frac{\tau}{\lambda_k}\ell^*_{\bar y}(\lambda \cdot)}(w).
    \]
\end{enumerate}
In Section \ref{sec:stability} we  show that this is natural  for classical data-fit terms.  We are  ready to prove our second main result for \ref{alg:fw_accel} about  error estimates under assumption \ref{assumption:approx_i} with $i=1$. Stability results for  type $2,3$ approximations are  deduced  noting that for these choices the error terms with $\varepsilon_{3,k}$ and $\varepsilon_{2,k}$ vanish, respectively, for every $k$.

\begin{theorem}[Error estimates for type 1 errors]\label{T:main estimate}
Assume that \ref{assumption:data-fit convex and coercive}-\ref{assumption:data-fit conditionned locally}, \ref{H:R strongly convex}-\ref{assumption:AR2}, \ref{assumption:parameter:stepsize}-\ref{assumption:parameter:technical} hold true.
Let $(\hat x_k)$, $(\hat u_k)$ be the sequences generated by \ref{alg:fw_accel} with noisy datum $\ynoisy$, and that \ref{assumption:approx_i} holds with $i=1$.
Then, we have the  stability estimate:
\begin{equation}   \label{error:estimate:iterates}
    (\forall k \geq 1) \quad
     t_k^2\frac{\sigma \tau}{2} \Vert \hat{x}_k - x^\dagger\Vert^2   
    \leq 
    C
	+   \sum_{j=1}^{k-1} t_{j+1}^2\varepsilon_{2,j}^2
	+ \frac{5}{2} \Big(\sum_{j=1}^{k-1} {t_{j+1}} \eps_{3,j} \Big)^2,
\end{equation}
where the constant $C$ is {defined} as 
\begin{equation}\label{def:constant rates}
C := 
\begin{cases}
    2\tau t_1^2 \Big({d}_1(\hat u_0) - \inf {d}_0\Big) + \Vert \hat u_0 - u^\dagger \Vert^2
    & \text{ if } q=1, \\
    2\tau t_1^2 \Big({d}_1(\hat u_0) - \inf {d}_0\Big) + \Vert \hat u_0 - u^\dagger \Vert^2 + 2\tau  \Lambda (1-\frac{1}{q}) {\gamma^{-1/(q-1)}} \Vert u^\dagger \Vert^{q/(q-1)}
    & \text{ if } q >1.
\end{cases}
\end{equation}
\end{theorem} 

\begin{proof}
Following the proof of Theorem~\ref{T:discrete convergence rates noiseless}, we define  the discrete energy function
\begin{equation}  \label{eq:discrete energy_IN}
\hat{\mathcal{E}}(k):=t_k^2 \Big( {d}_{{\lambda_k}}(\hat{u}_k)-\inf {d}_0\Big) + \frac{1}{2\tau}\|\hat{z}_k- u^\dagger\|^2,
\end{equation}
for  $k\geq 1$, where $u^\dagger\in \argmin {d}_0$ (so that $\inf {d}_0 =  {d}_0(u^\dagger)$) and $\hat{z}_k$ is defined as:
\begin{equation}  \label{eq:def_zk}
\hat{z}_k:=\hat{u}_{k-1} +t_k(\hat{u}_k-\hat{u}_{k-1}).
\end{equation}
Since $\hat u_{k+1} \approx^{\varepsilon_{1,k}}_{\text{\tiny 1}} \mathrm{prox}_{\tau \lambda_k^{-1}{\ell}_{\hat{y}}^*(\lambda_k~\cdot)}(\hat w_k - \tau \nabla \Rr_A(\hat w_k))$, using Definition \ref{def:approximation}, we have
\begin{equation}
    \xi_k := \frac{\hat w_k + e_k - \hat u_{k+1}}{\tau}, \qquad \xi_k - \nabla \Rr_A(\hat w_k) \in \partial_{\frac{\eps_{2,k}^2}{2\tau}} {\ell^*_{\hat y}}(\lambda_k \hat u_{k+1}),
\end{equation}
where $e_k \in \Hh$, $ \eps_{2,k}^2 + \eps_{3,k}^2 \leq \eps_{1,k}^2$ and $\Vert e_k \Vert \leq  \varepsilon_{3,k}$. 
Without loss of generality,
we can assume that
$\eps_{2,k}^2 + \eps_{3,k}^2 = \eps_{1,k}^2$.
Thus, by the descent lemma in
\cite[Lemma 4.1]{VillaSalzo2014} applied to ${d}_{\lambda_k} = \Rr_A + \lambda_k^{-1}{\ell}_{\hat{y}}^*(\lambda_k~\cdot)$, we derive
\begin{equation}
    {d}_{\lambda_k}(\hat u_{k+1}) \leq {d}_{\lambda_k}(u) + \langle \hat u_{k+1} - u , \xi_k \rangle + \frac{L}{2} \Vert \hat u_{k+1} - \hat w_k \Vert^2 + \frac{\eps_{2,k}^2}{2\tau},\qquad \forall u \in \Yy,
\end{equation}
where $L=\|A\|^2/\sigma^2$. Using the fact that $\tau L \leq 1$, rearranging and neglecting non-positive quantities, we obtain that for all $u \in \Yy$:
\begin{align} 
    {d}_{\lambda_k}(\hat u_{k+1}) & \leq 
    {d}_{\lambda_k}(u) 
    - \frac{1}{\tau}\Vert \hat u_{k+1} - \hat w_k \Vert^2
    + \langle \hat u_{k+1} - \hat w_k , \frac{e_k}{\tau} \rangle 
    + \langle \hat w_{k} - u , \xi_k \rangle \notag \\
    &
    + \frac{1}{2\tau} \Vert \hat u_{k+1} - \hat w_k \Vert^2 
    + \frac{\eps_{2,k}^2}{2\tau}  \notag \\
    & = 
    {d}_{\lambda_k}(u) 
    + \langle \hat w_{k} - u , \xi_k \rangle 
    - \frac{\tau}{2}\Vert  \frac{ \hat u_{k+1} - \hat w_k}{\tau} \Vert^2
    + \tau \langle  \frac{ \hat u_{k+1} - \hat w_k}{\tau} , \frac{e_k}{\tau} \rangle 
    + \frac{\eps_{2,k}^2}{2\tau} \notag \\
    & = 
    {d}_{\lambda_k}(u) 
    + \langle \hat w_{k} - u , \xi_k \rangle 
    - \frac{\tau}{2}\Vert  \xi_k \Vert^2
    + \frac{1}{2\tau}\left( \Vert e_k \Vert^2
    + \eps_{2,k}^2 \right)  \notag \\
    & \leq 
    {d}_{\lambda_k}(u) 
    + \langle \hat w_{k} - u , \xi_k \rangle 
    - \frac{\tau}{2}\Vert  \xi_k \Vert^2
    + \frac{\eps_{1,k}^2}{2\tau},  \label{eq:descent lemma noisy}
\end{align}
which can be seen as a noisy version of \eqref{eq:descent}.
We  divide the rest of the proof in three steps. Since the former ones are analogous to the calculations done in the error-free case, we will skip some of the details.
\paragraph{Step 1} We show that for every $k\geq 1$, there holds:
\begin{equation}  \label{eq:condition_IN}
    \hat{\Ee}(k+1) -\hat{\Ee}(k)
    \leq  
    t_{k+1} \Big({d}_{\lambda_k}(u^\dagger) -  \inf {d}_{0}\Big)  
    + \ \frac{t_{k+1}}{\tau}\langle e_k,  \hat{z}_k-u^\dagger \rangle
    + \frac{t_{k+1}^2}{2\tau}\varepsilon_{2,k}^2
\end{equation}
To prove this, we write the descent inequality \eqref{eq:descent lemma noisy} first for $u=\hat{u}_{k}$
 \begin{equation}  \label{descent:first_IN}
{d}_{\lambda_k}(\hat u_{k+1}) \leq 
    {d}_{\lambda_k}(\hat{u}_{k}) 
    + \langle \hat w_{k} - \hat{u}_{k} , \xi_k \rangle 
    - \frac{\tau}{2}\Vert  \xi_k \Vert^2
    + \frac{\eps_{1,k}^2}{2\tau},
\end{equation}
and then for $u=u^\dagger$
 \begin{equation}  \label{descent:second_IN}
{d}_{\lambda_k}(\hat u_{k+1}) \leq 
    {d}_{\lambda_k}(u^\dagger) 
    + \langle \hat w_{k} - u^\dagger , \xi_k \rangle 
    - \frac{\tau}{2}\Vert  \xi_k \Vert^2
    + \frac{\eps_{1,k}^2}{2\tau}.
\end{equation}
We now multiply \eqref{descent:first_IN} by $t_{k+1}-1$ and add it to \eqref{descent:second_IN}, thus getting:
\begin{align}   \label{ineq:1_IN}
t_{k+1} {d}_{{\lambda_k}}(\hat{u}_{k+1}) 
&\leq  (t_{k+1}-1){d}_{{\lambda_k}}(\hat{u}_{k}) +  {d}_{{\lambda_k}}(u^\dagger) \notag \\
&+ \langle  \xi_k, (t_{k+1}-1)(\hat{w}_k-\hat{u}_{k}) + \hat{w}_{k} - u^\dagger\rangle 
-\frac{t_{k+1}\tau}{2}\|  \xi_k\|^2 
+ \frac{t_{k+1}}{2\tau} \varepsilon_{1,k}^2  
\end{align}
We  apply the property $(t_{k+1}-1)(\hat{w}_k-\hat{u}_{k}) + \hat{w}_{k} = \hat z_k$ (see \eqref{eq:def_zk:noisefree}) and  write \eqref{ineq:1_IN} as
\begin{eqnarray*}
    & & 
    t_{k+1} ( {d}_{{\lambda_k}}(\hat{u}_{k+1}) -  {d}_{{\lambda_k}}(u^\dagger)) \\
    &\leq  & 
    (t_{k+1}-1) ({d}_{{\lambda_k}}(\hat{u}_{k}) -  {d}_{{\lambda_k}}(u^\dagger)) 
    + \frac{1}{2\tau t_{k+1}} \left(\Vert \hat z_k - u^\dagger \Vert^2
    - \Vert \hat z_k - u^\dagger - \tau t_{k+1}\xi_k \Vert^2 \right)
    + \frac{t_{k+1}}{2\tau} \varepsilon_{1,k}^2.
\end{eqnarray*}
From the identity $- \tau t_{k+1}\xi_k = \hat z_{k+1} - \hat z_k - t_{k+1}e_k$,
we deduce:
\begin{eqnarray*}
    & & 
    t_{k+1} ( {d}_{{\lambda_k}}(\hat{u}_{k+1}) -  {d}_{{\lambda_k}}(u^\dagger)) 
    + \frac{1}{2\tau t_{k+1}} \Vert \hat z_{k+1} - u^\dagger \Vert^2\\
    &\leq  & 
    (t_{k+1}-1) ({d}_{{\lambda_k}}(\hat{u}_{k}) -  {d}_{{\lambda_k}}(u^\dagger)) 
    + \frac{1}{2\tau t_{k+1}} \Vert \hat z_k - u^\dagger \Vert^2
    + \frac{1}{\tau} \langle \hat z_{k+1} - u^\dagger , e_k \rangle
    + \frac{t_{k+1}}{2\tau}\left( \varepsilon_{1,k}^2 - \|e_k\|^2\right). \\
        &=  & 
    (t_{k+1}-1) ({d}_{{\lambda_k}}(\hat{u}_{k}) -  {d}_{{\lambda_k}}(u^\dagger)) 
    + \frac{1}{2\tau t_{k+1}} \Vert \hat z_k - u^\dagger \Vert^2
    + \frac{1}{\tau} \langle \hat z_{k+1} - u^\dagger , e_k \rangle
    + \frac{t_{k+1}}{2\tau} \varepsilon_{2,k}^2. 
\end{eqnarray*}
We now multiply everything by $t_{k+1}$, re-arrange and get
\begin{eqnarray*}  \label{descent:third2_IN}
    & & 
    t^2_{k+1} \Big( {d}_{{\lambda_k}}(\hat{u}_{k+1}) -  {d}_{{\lambda_k}}(u^\dagger) \Big)   + \frac{1}{2\tau}
    \|\hat{z}_{k+1}-u^\dagger\|^2   \\
    &\leq&
    t^2_k \Big({d}_{{\lambda_k}}(\hat{u}_{k}) -  {d}_{{\lambda_k}}(u^\dagger)\Big) +   (t^2_{k+1}-t_{k+1} - t_{k}^2) \Big( {d}_{{\lambda_k}}(\hat{u}_{k}) -  {d}_{{\lambda_k}}(u^\dagger) \Big) 
    + \frac{1}{2\tau } \|\hat{z}_{k}-u^\dagger\|^2  \notag \\
    & &
    + \ \frac{t_{k+1}}{\tau}\langle e_k,  \hat{z}_{k+1}-u^\dagger \rangle
    + \frac{t_{k+1}^2}{2\tau}\varepsilon_{2,k}^2. 
\end{eqnarray*}
Using now that $d_{\lambda_k}(\hat u_k) \geq \inf d_0$ (see Lemma  \ref{L:functional properties}\ref{L:functional properties:decreasing function sequence}), adding and subtracting in the parentheses the term $\inf {d}_0$ and after recalling the definition of $\mathcal{E}$ in \eqref{eq:discrete energy_IN}, we get:
\begin{multline}
    \hat{\Ee}(k+1) + (t^2_k + t_{k+1}-t^2_{k+1})  \Big({d}_{{\lambda_k}}(u_k) -  \inf {d}_0\Big)  \\
    \leq  
    \hat{\Ee}(k) 
    + t_{k+1} \Big({d}_{{\lambda_k}}(u^\dagger) -  \inf {d}_0\Big)  
    + \ \frac{t_{k+1}}{\tau}\langle e_k,  \hat{z}_k-u^\dagger \rangle
    + \frac{t_{k+1}^2}{2\tau}\varepsilon_{2,k}^2,
\end{multline}  
whence we deduce condition \eqref{eq:condition_IN} since $t^2_k + t_{k+1}-t^2_{k+1}\geq 0$ and ${d}_{{\lambda_k}}(u^\dagger) -  \inf {d}_0 \geq 0$ (see \eqref{eq:properties_sequences_alphak_tk}). 
Iterating recursively \eqref{eq:condition_IN},  Cauchy-Schwartz inequality yields \begin{align}  \label{eq:decay_complete_IN}
    \hat{\mathcal{E}}(k) 
    \leq 
    \hat{\mathcal{E}}(1) 
    + \sum_{j=1}^{k-1}t_{j+1}\Big( {d}_{\lambda_j}(u^\dagger)-\inf {d}_0 \Big)
    + \sum_{j=1}^{k-1} \frac{t_{j+1}}{\tau} \eps_{3,j} \Vert \hat{z}_{j+1}-u^\dagger\Vert 
    + \sum_{j=1}^{k-1}\frac{t_{j+1}^2}{2\tau}\varepsilon_{2,j}^2,
\end{align}
which will be used to deduce the following stability estimate.
Next, we  study separately the sums appearing on the right-hand side of \eqref{eq:decay_complete_IN}.
\paragraph{Step 2} 
For the first term in \eqref{eq:decay_complete_IN}, following  the proof of Theorem \ref{T:discrete convergence rates noiseless}, we get
$$
\sum_{j=1}^{k-1} t_{j+1} \Big( {d}_{\lambda_j}(u^\dagger)-\inf {d}_0 \Big) \leq  c  \sum_{j=1}^{k-1} t_{j+1}\lambda_{\lambda_j}^{\frac{1}{q-1}} \leq c \Lambda < +\infty.
$$
where $c$ is defined in \eqref{proof:ccd3}, and $\Lambda$ is finite thanks to assumption \ref{assumption:parameter:relaxation}.
 
\paragraph{Step 3} To bound the second sum  in \eqref{eq:decay_complete_IN}, we  observe that by definition $\hat \Ee(k) \geq \frac{1}{2\tau} \Vert \hat z_k - u^\dagger \Vert^2$.
Then, we set $C= 2 \tau(\hat \Ee(1) + c\Lambda)$ and we derive
\begin{equation}
    \Vert \hat z_k - u^\dagger \Vert^2
    \leq 
     C
    + \sum_{j=1}^{k-1}{t_{j+1}^2}\varepsilon_{2,j}^2
    + 2 \sum_{j=1}^{k-1} {t_{j+1}} \eps_{3,j} \Vert \hat{z}_{j+1}-u^\dagger\Vert 
    .
\end{equation}
Lemma \ref{lemma:attouch (sequences)} applied to $a_k=\Vert \hat z_k - u^\dagger \Vert$, $b_k=2t_{k+1}\eps_{3,k}$, $c_{k-1}=C + \sum_{j=1}^{k-1}{t_{j+1}^2}\varepsilon_{2,j}^2$ implies
\begin{align}
\sum_{j=1}^{k-1} t_{j+1} \eps_{3,j} \Vert \hat{z}_{j+1}-u^\dagger\Vert 
\leq 
\Big(\sum_{j=1}^{k-1} {t_{j+1}} \eps_{3,j} \Big) \Big(\sqrt{C + \sum_{j=1}^{k-1} t_{j+1}^2\varepsilon_{2,j}^2 } + 2 \sum_{j=1}^{k-1} t_{j+1}\varepsilon_{3,j}\Big). \notag
\end{align}
Combining altogether in \eqref{eq:decay_complete_IN}, we deduce
\begin{align}  \label{eq:conclusion}
    \hat{\mathcal{E}}(k)  
    & \leq \frac{C}{2\tau}  
    + \sum_{j=1}^{k-1}\frac{t_{j+1}^2}{2\tau}\varepsilon_{2,j}^2 
    + \frac{1}{\tau}\Big(\sum_{j=1}^{k-1} {t_{j+1}} \eps_{3,j} \Big) \Big(\sqrt{C + \sum_{j=1}^{k-1} t_{j+1}^2\varepsilon_{2,j}^2 } + 2 \sum_{j=1}^{k-1} t_{j+1}\varepsilon_{3,j}\Big)   
    .  
\end{align} 
Young's inequality applied to the product  in the right hand side of \eqref{eq:conclusion} yields
\begin{eqnarray*}
    && 
   \frac{1}{\tau}\Big(\sum_{j=1}^{k-1} {t_{j+1}} \eps_{3,j} \Big) \left(\sqrt{C + \sum_{j=1}^{k-1} t_{j+1}^2\varepsilon_{2,j}^2 } + 2 \sum_{j=1}^{k-1} t_{j+1}\varepsilon_{3,j}\right)   \\
	& = &
	\frac{1}{\tau} \Big(\sum_{j=1}^{k-1} {t_{j+1}} \eps_{3,j} \Big) \left(\sqrt{C + \sum_{j=1}^{k-1} t_{j+1}^2\varepsilon_{2,j}^2 } \right) +
	\frac{2}{\tau}
	\Big(\sum_{j=1}^{k-1} {t_{j+1}} \eps_{3,j} \Big)^2
     \\
	& \leq & 
\frac{5}{2\tau}
	\left(\sum_{j=1}^{k-1} {t_{j+1}} \eps_{3,j} \right)^2 + \frac{C}{2\tau} + \frac{1}{2\tau} \sum_{j=1}^{k-1} t^2_{j+1}\varepsilon^2_{3,j}
\end{eqnarray*}
From \eqref{eq:conclusion} we thus obtain
\begin{equation}  \label{error:energy:stability:final}
	\hat \Ee(k)
	\leq
	\frac{C}{\tau}
	+ \frac{1}{\tau}  \sum_{j=1}^{k-1} t_{j+1}^2\varepsilon_{2,j}^2
	+  \frac{5}{2\tau} \Big(\sum_{j=1}^{k-1} {t_{j+1}} \eps_{3,j} \Big)^2.
\end{equation}
To conclude, we use Lemma \ref{L:functional properties}\ref{L:functional properties:from dual value to primal iterates} and deduce
\begin{equation}
	\hat \Ee(k) 
	\geq t_k^2(d_{\lambda_k}(\hat u_k) - \inf d_0)
	\geq t_k^2(d_0(\hat u_k) - \inf d_0)
	\geq   \frac{t_k^2 \sigma}{2} \Vert \hat x_k - x^\dagger \Vert^2,
\end{equation} 
that combined with \eqref{error:energy:stability:final} gives the desired stability estimate \eqref{error:estimate:iterates}.
\end{proof}

\subsection{Early stopping}\label{sec:earlystopping}
We provide early stopping results guaranteeing the  iterative regularization properties  of  \ref{alg:fw_accel}.
 Different convergence rates are obtained depending on type of approximation considered (see Definition \ref{def:approximation}). The results follow from  Theorem \ref{T:main estimate} considering assumption \ref{assumption:approx_i} for  $i\in\left\{1,2,3\right\}$.

\begin{theorem}[Early stopping for type 1 errors]\label{T:main result type 1}
Assume that \ref{assumption:data-fit convex and coercive}-\ref{assumption:data-fit conditionned locally}, \ref{H:R strongly convex}-\ref{assumption:AR2}, \ref{assumption:parameter:stepsize}-\ref{assumption:parameter:technical} hold true, 
and suppose that $\lambda_k = \Theta( k^{-\theta})$ with $\theta > 2(q-1)$. 
Let $(\hat x_k)$ be the sequence generated by \ref{alg:fw_accel} with noisy datum $\ynoisy$, and assume that \ref{assumption:approx_i} holds with $i=1$, $\eps_{2,k} = O(\delta \lambda_k^{-r_2})$, $\eps_{3,k} = O(\delta \lambda_k^{-r_3})$ for some $\delta >0$ and $r_2,r_3 \geq 0$.
Set:
\begin{equation}
\alpha := \max\left\{\frac{2}{{3}+2r_2 \theta}, \frac{1}{2 + r_3\theta } \right\}.
\end{equation}
Then, any early stopping rule with $k(\delta)= \Theta (\delta^{- \alpha})$ verifies:
\begin{equation}
    \Vert \hat{x}_{k(\delta)} - x^\dagger\Vert = O\left( \delta^{\alpha} \right), \quad \text{for }\delta\searrow 0. 
\end{equation}
\end{theorem}

\begin{proof}
We apply estimate \eqref{error:estimate:iterates} from Theorem \ref{T:main estimate}. After substituting the expression for $\varepsilon_{2,k}$ and $\varepsilon_{3,k}$, we get:
\begin{equation}
    t_k^2 \Vert \hat{x}_k - x^\dagger\Vert^2   
    = 
    O \Big(1 + \sum_{j=1}^{k-1} {t^2_{j+1}} \eps_{2,j}^2 + \Big(\sum_{j=1}^{k-1} {t_{j+1}} \eps_{3,j} \Big)^2 \Big)
	= 
	O(1 + \delta^2 k^{3+2r_2\theta}+ \delta^2 k^{4+2r_3\theta}).
\end{equation}
In correspondence with the stopping time $k(\delta)$, and using the fact that $t_{k(\delta)} = \Theta (k(\delta))$, we deduce from above:
\begin{align}
\Vert \hat{x}_{k(\delta)} - x^\dagger\Vert^2 & =
O\left(\delta^{2\alpha} + \delta^{2 - \alpha(1+2r_2\theta)} + \delta^{2 - 2 \alpha (1+r_3\theta) }   \right) \notag \\
&
= 
O\left(\delta^{\min\{2\alpha ; 2 - \alpha(1+2r_2\theta), 2 - 2 \alpha (1+r_3\theta)  \}}  \right).
\end{align}
Let us now define $\beta := \min \{ \frac{1}{2}+r_2 \theta ; 1 + r_3 \theta \}$. We easily see that 
\begin{equation}
\min\{ 2 - \alpha(1+2r_2\theta) ; 2 - 2 \alpha (1+r_3\theta)  \}
= 
2 - 2 \alpha \beta,
\end{equation}
so that 
$\min\{2\alpha, 2 - \alpha(1+2r_2\theta), 2 - 2 \alpha (1+r_3\theta)  \}
=
\min\{
2 \alpha, 2 - 2 \alpha \beta
\}$,
which is maximal for $\alpha = \frac{1}{1+\beta}$.
\end{proof}
The analogous results for errors of type $2$ and $3$ are  straightforward.
\begin{theorem}[Early stopping for type 2 errors]\label{T:main result type 2}
Assume that the assumptions \ref{assumption:data-fit convex and coercive}-\ref{assumption:data-fit conditionned locally}, \ref{H:R strongly convex}-\ref{assumption:AR2}, \ref{assumption:parameter:stepsize}-\ref{assumption:parameter:technical} hold true, and suppose that $\lambda_k = \Theta( k^{-\theta})$ with $\theta > 2(q-1)$. 
Let $(\hat x_k)$ be the sequence generated by \ref{alg:fw_accel} with noisy datum $\ynoisy$, and assume that \ref{assumption:approx_i} holds with $i=2$, $\eps_{2,k} = O(\delta \lambda_k^{-r_2})$ for some $\delta >0$ and $r_2 \geq 0$.
Then, any early stopping rule with $k(\delta) = \Theta (\delta^{-\frac{2}{3 + 2\theta r}})$ verifies:
\begin{equation}
    \Vert \hat{x}_{k(\delta)} - x^\dagger\Vert = O\left( \delta^{\frac{2}{3 + 2\theta r}} \right), \quad \text{for }\delta\searrow 0. 
\end{equation}
\end{theorem}

\begin{proof}
For  type 2 approximation  \eqref{error:estimate:iterates} $\eps_{3,k} \equiv 0$, and we get 
\begin{equation}
    t_k^2 \Vert \hat{x}_k - x^\dagger\Vert^2   
    = 
    O \Big(1 + \sum_{j=1}^{k-1} {t^2_{j+1}} \eps_{2,j}^2 \Big)
	= 
    O \Big(1 + \sum_{j=1}^{k-1}  \delta^2 j^{2+2r\theta} \Big)
	=
	O(1 + \delta^2 k^{3+2r\theta}).
\end{equation}
In correspondence with any stopping time $k(\delta) = \Theta( \delta^{-\alpha})$, we thus have:
\begin{equation}
\Vert \hat{x}_{k(\delta)} - x^\dagger\Vert^2 
=
O\left(k(\delta)^{-2} + \delta^2 k(\delta)^{1+2r\theta} \right)
= 
O\left(\delta^{2\alpha} + \delta^{2 - \alpha(1+2r\theta)}  \right).
\end{equation}
The term on the right-hand side is minimized when
$\alpha = \frac{2}{3 + 2\theta r}$.
\end{proof}

\begin{theorem}[Early stopping for type 3 errors]\label{T:main result type 3}
Assume that the assumptions \ref{assumption:data-fit convex and coercive}-\ref{assumption:data-fit conditionned locally}, \ref{H:R strongly convex}-\ref{assumption:AR2}, \ref{assumption:parameter:stepsize}-\ref{assumption:parameter:technical} hold true, and suppose that $\lambda_k = \Theta( k^{-\theta})$ with $\theta > 2(q-1)$. 
Let $(\hat x_k)$ be the sequence generated by \ref{alg:fw_accel} with noisy datum $\ynoisy$, and assume that \ref{assumption:approx_i} holds with $i=3$ with $\eps_{3,k} = O(\delta \lambda_k^{-r_3})$ for some $\delta >0$ and $r_3 \geq 0$.
Then, any early stopping rule with $k(\delta) = \Theta( \delta^{-\frac{1}{2+\theta r}})$ verifies:
\begin{equation}
    \Vert \hat{x}_{k(\delta)} - x^\dagger\Vert = O\left( \delta^{\frac{1}{2+\theta r}} \right), \quad \text{for }\delta\searrow 0. 
\end{equation}
\end{theorem}

\begin{proof}
{Assuming type 3} errors means that in the estimate \eqref{error:estimate:iterates} $\eps_{2,k} \equiv 0$, so that:
\begin{equation}
    t_k^2 \Vert \hat{x}_k - x^\dagger\Vert^2   
    = 
    O(1) 
	+ O \Big(\sum_{j=1}^{k-1} {t_{j+1}} \eps_{3,j} \Big)^2
	= 
	O(1) 
	+ O \Big(\sum_{j=1}^{k-1}  \delta j^{1+r \theta} \Big)^2
	=
	O(1 + \delta^2 k^{4+2r\theta}).
\end{equation}
In correspondence with the stopping time $k(\delta)= \Theta( \delta^{-\alpha})$, we thus deduce:
\begin{equation}
\Vert \hat{x}_{k(\delta)} - x^\dagger\Vert^2 
= 
O\left(k(\delta)^{-2} + \delta^2 k(\delta)^{2+2r\theta} \right)
=
O\left(\delta^{2\alpha} + \delta^{2 - 2 \alpha (1+r\theta) }  \right).
\end{equation}
The term on the right-hand side is minimal whenever 
$\alpha = \frac{1}{2 + \theta r}$.
\end{proof}

\section{Applications to specific data-fit terms} \label{sec:stability}

We next apply the results from Section \ref{sec:earlystopping} to relevant data-fit terms. The following definition is useful.

 \begin{definition}[$\delta$-perturbation]   \label{def:perturbation}
 For given $\bar{y},~\hat{y}\in \mathcal{Y}$ and $\delta\in\R_{++}$, we say that $\hat{y}$ is a $\delta$-perturbation of $\bar{y}$ according to $\ell$ if:
 $$
\ell_{\hat{y}}(\bar{y})=\ell(\bar{y},\hat{y})\leq \delta^q, 
 $$
 where $q\in [1,+\infty)$ is the conditioning exponent appearing in  \ref{assumption:data-fit conditionned locally}.
 \end{definition}
We show that a $\delta$-perturbation in the data corresponds to a proximal mapping of $\ell_{\hat{y}}^*$ approximating the corresponding proximal mapping of $\ell_{\bar{y}}^*$ in the sense of Definition \ref{def:approximation}  and with some precision $\varepsilon(\delta)$ depending on the noise level $\delta$. 
 
\subsection{Additive data-fit terms}  \label{sec:distance_based}

For additive data-fit functions (see Example \ref{R:data-fit function list}),  a $\delta$-perturbation corresponds to a type $3$ approximation of the proximal mapping.

\begin{proposition}[Additive data-fit terms lead to type $3$ errors]\label{P:additive loss type 3} Let $\mathcal{N}\in\Gamma_0(\mathcal{Y})$ and
assume that $\ell_{y_2}(y_1)=\mathcal{N}(y_2-y_1)$, for every $(y_1,y_2)\in \mathcal{Y}^2$.
For given $(\delta,\tau,\lambda) \in(0,+\infty)^3$, let  $\hat{y}\in \mathbb{B}(\bar{y},\varrho)$ be a $\delta$-perturbation of  $\bar{y}$ in the sense of Definition \ref{def:perturbation}.
Then:
$$
(\forall z \in \Yy) \quad \hat{p}=\mathrm{prox}_{\frac{\tau}{\lambda}\ell^*_{\hat{y}}(\lambda\cdot)}(z) ~\approx^{\varepsilon}_{\text{\tiny 3}}~ \bar{p}=\mathrm{prox}_{\frac{\tau}{\lambda}\ell^*_{\bar{y}}(\lambda\cdot)}(z). 
$$
with precision $\varepsilon=\tau\delta(q/\gamma)^{1/q}$ and where $q\geq 1$ and $\gamma>0$ are the conditioning parameters appearing in assumption \ref{assumption:data-fit conditionned locally}.
\end{proposition}

\begin{proof}
We need to find $e\in \mathcal{Y}$ and $\varepsilon \geq 0$ such that $\|e\|\leq \varepsilon$ and:
\begin{equation}   \label{eq:proof_distance_condition}
\frac{z+e-\hat{p}}{\tau} \in \frac{1}{\lambda} \partial \ell_{\bar{y}}^*(\lambda \cdot)(\hat{p}).
\end{equation}
Due to the special form of the data-fit we start noting that for any $u\in \mathcal{Y}$ we have
$$
\ell^*_{\bar{y}} (u) =\mathcal{N}^*(u) + \langle \bar{y}, u\rangle,
$$
and the same holds for $\ell^*_{\ynoisy}$. Then
\begin{equation} \label{eq:union_subdifferentials}
\partial \ell^*_{\hat{y}}(\lambda \cdot)(\hat p) = \lambda \partial \ell^*_{\hat{y}}(\lambda \hat{p}) = \lambda \partial \Big( \mathcal{N}^* + \langle \hat{y}, \cdot\rangle \Big) (\lambda \hat{p})= \lambda \partial \mathcal{N}^*(\lambda \hat{p}) + \lambda \hat{y}.
\end{equation}
By definition of $\hat{p}$ we have
$({z-\hat{p}})/{\tau} \in ({1}/{\lambda})\partial \ell^*_{\hat{y}}(\lambda \cdot)(\hat p) = \partial \mathcal{N}^*(\lambda \hat{p}) +  \hat{y}$,
which, by simple algebraic manipulations, entails  the required condition \eqref{eq:proof_distance_condition}, since:
$$
\frac{z-\hat{p}}{\tau} \in  \partial \mathcal{N}^*(\lambda \hat{p}) + \bar{y} + (\hat{y}-  \bar{y})
 \iff
\frac{z - \hat p + \tau(\bar{y}-\hat{y})}{\tau}\in \partial \mathcal{N}^*(\lambda \hat{p}) + \bar{y} = \frac{1}{\lambda}\partial \ell^*_{\bar{y}}(\lambda \cdot)(\hat p). 
$$
By setting $e=\tau(\bar{y}-\hat{y})$, we can find the required $\varepsilon$ combining the local $q$-conditioning of the function $\ell_{\bar y}$ on $\mathbb{B}(\bar{y},\varrho)$ assumed in \ref{assumption:data-fit conditionned locally} with the $\delta$-perturbation  assumption:  
\[
\|e\| = \tau\|\bar{y}-\hat{y}\|\leq  \tau\left(\frac{q}{\gamma}\ell(\hat{y},\bar{y})\right)^{1/q}\leq \tau\left(\frac{q}{\gamma}\right)^{1/q}
\delta=:\varepsilon,
\]
where $\gamma>0$ and $q\geq 1$ are the conditioning parameters.
We can thus conclude that $\hat{p}$ is a $\varepsilon$-approximation of $\bar{p}$ with precision $\varepsilon$, as required.
 \end{proof}

Thanks to Proposition \ref{P:additive loss type 3}, we can derive
early-stopping result by   applying Theorem \ref{T:main result type 3} to additive data-fit terms with  the above choice of $\varepsilon$.

\begin{corollary}[Early stopping for additive data-fit terms]\label{C:main result type 3} Let $\mathcal{N}\in\Gamma_0(\mathcal{Y})$ and set
$\ell_{y_2}(y_1)=\mathcal{N}(y_2-y_1)$, for every $(y_1,y_2)\in\mathcal{Y}^2$
Assume that the assumptions \ref{assumption:data-fit convex and coercive}-\ref{assumption:data-fit conditionned locally}, \ref{H:R strongly convex}-\ref{assumption:AR2}, \ref{assumption:parameter:stepsize}-\ref{assumption:parameter:technical} hold, and  that $\lambda_k = \Theta( k^{-\theta})$ with $\theta > 2(q-1)$. 
Let $(\hat x_k)$ be the sequence generated by \ref{alg:fw_accel} with  $\hat y \in \mathbb{B}(\bar y,\varrho)$, such that $\hat y$ is a $\delta$-perturbation of $\bar y$.
Then, any early stopping rule with $k(\delta) = \Theta( \delta^{-1/2})$ verifies:
\begin{equation}  \label{eq:stab_estimate_distance_based}
 \|  \hat{x}_{k(\delta)} - x^\dagger \|=O(\delta^{\frac{1}{2}}), \quad \text{for }\delta\searrow 0. 
\end{equation}
\end{corollary}

\begin{remark}[Optimality of the rates] \label{remark:optimal rates}
The convergence rate in  \eqref{eq:stab_estimate_distance_based} is optimal for regularization methods for additive data-fit terms
\cite{EngHanNeu96}. Among  inertial algorithms, optimal convergence rates for different regularizers but only quadratic data-fit terms have been proved  in 
\cite{Neubauer2016,MatRosVilVu17}.  For  more general additive data-fits (e.g.  the $\ell^1$-norm, see Example \ref{R:data-fit function list}), in 
\cite{BenningBurger2011} the authors  prove a  rate  $O(\delta^{1/2})$ on the Bregman distance, which is different from \eqref{eq:stab_estimate_distance_based}. 
To our knowledge, our result  is the first  showing optimal convergence rates for iterative regularization methods for general data-fit term improving the estimates in \cite{Garrigos2017}  that showed  a rate $O(\delta^{1/3})$.
\end{remark}

\begin{remark}[Different growth for $t_k$]\label{cor:general_tk}
As noted in  Remark \ref{R:different choices t_k}, if we replace $t_k=\Theta(t_k)$ by $t_k=\Theta(k^\beta)$, then  $\beta \leq 1$, and $\beta=1$ gives the fastest  convergence rate for true datumum $\bar y$.
Corollary \ref{C:main result type 3} implies  that also for noisy data $\hat y$, any stopping rule with $k(\delta) = \Theta( \delta^{-1/(1+\beta)})$ verifies
$\| \hat{x}_{k(\delta)} - x^\dagger \| = O(\delta^{ \frac{\beta}{\beta+1}})$ for $\delta\searrow 0,$
where  again the best rate is achieved for  $\beta = 1$.
\end{remark}
 
\subsection{KL divergence}  \label{sec:KLerrors}
 
We  consider the  Kullback-Leibler (KL) divergence as an example of non-additive data-fit term. KL divergence is often used to model data corrupted by Poisson noise. We  show that the KL divergence $\delta$-perturbations lead to type $2$ approximations. We recall that the  KL divergence is locally $2$-conditionned (see Example \ref{R:data-fit function list}).

\begin{proposition}  \label{prop:KL_approx}
Assume that, $\ell_{y_2}(y_1)=\mathrm{KL}(y_2;y_1)$ for every $(y_1,y_2)\in\mathcal{Y}^2$.
For $(\delta,\tau,\lambda) \in (0,+\infty)^3$, let $\hat{y}\in \mathbb{B}(\bar{y},\varrho)$ be a $\delta$-perturbation of  $\bar{y}$.
Then
$$
(\forall z \in \Yy) \quad \hat{p}=\mathrm{prox}_{\frac{\tau}{\lambda}\ell^*_{\hat{y}}(\lambda\cdot)}(z) ~\approx^{\varepsilon}_{\text{\tiny 2}}~ \bar{p}=\mathrm{prox}_{\frac{\tau}{\lambda}\ell^*_{\bar{y}}(\lambda\cdot)}(z). 
$$

with $\varepsilon=\sqrt{2\tau}\delta/\lambda$.
\end{proposition}

\begin{proof}
It is enough to prove that for all $z\in \mathcal{Y}$
\begin{equation} \label{definition:KL_thesis1}
\frac{\lambda(z-\hat{p})}{\tau} \in \partial_{\frac{\lambda\varepsilon^2}{2\tau}} \mathrm{KL}^*_{\bar y}(\lambda \cdot)(\hat{p}) = \lambda \partial_{\frac{\lambda\varepsilon^2}{2\tau}} \mathrm{KL}^*_{\bar y}(\lambda\hat{p}),\quad\Longleftrightarrow\quad\frac{z-\hat{p}}{\tau}\in\partial_{\frac{\lambda\varepsilon^2}{2\tau}} \mathrm{KL}^*_{\bar y}(\lambda\hat{p}).
\end{equation} 
We set $x=({z-\hat{p}})/{\tau}\in\mathcal{Y}$ and consider the function $g:\mathcal{Y} \to \R^d\cup \left\{ + \infty \right\}$ defined by
\begin{equation}   \label{eq:def_g}
g(w) = \frac{\mathrm{KL}_{\bar y}}{\lambda}(w),\quad\text{for all }w \in \mathcal{Y}.
\end{equation}
By standard property of convex conjugates we have that for any $u\in \mathcal{Y}$
\begin{equation} \label{eq:prop_conj}
g^*(u) = \left(  \frac{\mathrm{KL}_{\bar y}}{\lambda} \right)^*(u) = \frac{1}{\lambda} \mathrm{KL}^*_{\bar y}(\lambda u).
\end{equation}
We now claim that  $x\in\partial_{\frac{\lambda\varepsilon^2}{2\tau}} g^*(\hat{p})$. To show that, we apply the Young-Fenchel inequality of Lemma \ref{lemma:characterisation_eps} to $g$ with  $x^*=\hat{p}$. Our objective is thus to show that:
\[
g(x) + g^*(\hat{p}) \leq \langle x, \hat{p} \rangle + \frac{\lambda\varepsilon^2}{2\tau}, 
\]
which, by definitions \eqref{eq:def_g} and \eqref{eq:prop_conj} and upon multiplication by $\lambda$, coincides with: 
\begin{equation}   \label{definition:KL_YF}
\mathrm{KL}_{\bar y}(x) + \mathrm{KL}^*_{\bar y}(\lambda \hat{p})\leq \langle x, \lambda \hat{p} \rangle +\frac{\lambda^2\varepsilon^2}{2\tau}.
\end{equation}
Using expressions \eqref{KLdefinition} and \eqref{definition:KL*} for KL and its convex conjugate, we express the sum on the left hand side of \eqref{definition:KL_YF} as:
\begin{equation}  \label{eq:calculationsKL}
\mathrm{KL}_{\bar{y}}(x) + \mathrm{KL}^*_{\bar{y}}(\lambda \hat{p}) = \sum_{i=1}^d \Big(\bar{y}_i \log \frac{\bar{y}_i}{x_i} - \bar{y}_i + x_i - \bar{y}_i\log(1-\lambda \hat{p}_i)\Big).
\end{equation}
Furthermore, by definition of $\hat{p}$, we have that component-wise there holds:
$$
\frac{\lambda}{\tau} ( z_i - \hat{p}_i ) \in \lambda\partial \mathrm{kl}_{\hat{y}_i}^*(\lambda \hat{p}_i) \quad \Longleftrightarrow \quad x_i \in \partial \mathrm{kl}^*_{\hat{y}_i}(\lambda \hat{p}_i),
$$
which, 
since $\mathrm{kl}_{\hat{y}_i}^*$ is differentiable (see formula \eqref{definition:KL*}), entails that for every $i=1,\ldots,d$ the element $x_i$ can be written as
$x_i = {\hat{y}_i}/{1-\lambda \hat{p}_i}$.
Substitute this expression in the formula  \eqref{eq:calculationsKL} to derive
\begin{align}
\mathrm{KL}_{\bar{y}}(x) + \mathrm{KL}^*_{\bar{y}}(\lambda \hat{p}) &= \sum_{i=1}^d \underbrace{ \bar{y}_i\log\bar{y}_i - \bar{y}_i\log \hat{y}_i -\bar{y}_i + \hat{y}_i}_{\mathrm{kl}(\bar{y}_i;\hat{y}_i)} \notag \\
& +  \cancel{\bar{y}_i\log(1-\lambda \hat{p}_i)}+  \underbrace{\big({ \hat{y}_i}/({1-\lambda \hat{p}_i})\big)}_{x_i}\lambda \hat{p}_i -\cancel{\bar{y}_i\log(1-\lambda \hat{p}_i)}\notag\\
&= \mathrm{KL}_{\bar{y}}(\hat{y}) + \langle x, \lambda \hat{p} \rangle \notag\\ 
&\leq \delta^2  + \langle x, \lambda \hat{p} \rangle,
\end{align}
where the last inequality follows from the perturbation assumption $\mathrm{KL}_{\bar{y}}(\hat{y})\leq \delta^2$.
We thus get \eqref{definition:KL_YF} by choosing $\varepsilon=\sqrt{2\tau}\delta/\lambda$, which concludes the proof.
\end{proof}
From Proposition \ref{prop:KL_approx} and Theorem \ref{T:main result type 2},  we  derive     stopping rules for the  KL divergence.
\begin{corollary}[Early stopping for Kullback-Leibler divergence]
Let $\ell_{y_2}(y_1)=\mathrm{KL}(y_2;y_1)$ for every $(y_1,y_2)\in\mathcal{Y}^2$.
Assume that the assumptions \ref{assumption:data-fit convex and coercive}-\ref{assumption:data-fit conditionned locally}, \ref{H:R strongly convex}-\ref{assumption:AR2}, \ref{assumption:parameter:stepsize}-\ref{assumption:parameter:technical} hold true, and suppose that $\lambda_k = \Theta( k^{-\theta})$ with $\theta > 2$. 
Let $(\hat x_k)$ be the sequence generated by \ref{alg:fw_accel} given $\hat y$, such that $\hat y$ is a $\delta$-perturbation of $\bar y$ in the sense of Definition \ref{def:perturbation}.
Then, any early stopping rule with $k(\delta)= \Theta( \delta^{-\frac{2}{3 + 2\theta}})$ verifies
\begin{equation}  \label{eq:stab_estimate_KL}
 \|  \hat{x}_{k(\delta)} - x^\dagger \|=O(\delta^{\frac{2}{3+2\theta}}), \quad \text{for }\delta\searrow 0.
\end{equation}
\end{corollary} 

\begin{remark}   \label{remark:KL rates}
It is hard to assess the quality of the rate in \eqref{eq:stab_estimate_KL} since the  the notion of optimality in \cite{EngHanNeu96} only applies to additive noise. In the context of Bregman divergences, a similar analysis has been pursued in \cite[Section 4.2, estimate (4.3)]{BenningBurger2011}. The estimate in \cite{BenningBurger2011} leads to a rate of order $\delta^{1/4}$ for suitable choices of the regularization parameter. In comparison, our estimate \eqref{eq:stab_estimate_KL} is sharper  and more explicit. As for additive data-fit terms, the use of inertia improves the rates in \cite{Garrigos2017}.
\end{remark}

\begin{remark}[The Kullback-Leibler divergence does not lead to type 3 errors]\label{R:KL not type 3 error}
The convergence rates for additive data-fit terms proved in Corollary \ref{C:main result type 3} are better than the rate for the KL divergence, due to the fact that for the KL divergence we proved that $\delta$-perturbations correspond to type 2 errors, instead of type 3 errors.
Indeed, Lemma \ref{L:KL not type 3 error} in the Appendix shows that the error in the evaluation of proximal points for the KL divergence can not be cast in a type 3 approximation.
\end{remark}

\section{Conclusions and outlook}

In this paper we proposed an inertial dual diagonal method to solve inverse problems for a wide class of data-fit and regularization terms, possibly corrupted by noise.  On the one hand we established convergence results both for continuous and discrete dynamics. On the other hand we derived stability results and corresponding stopping rules, characterizing the regularization properties of the proposed method. A number of open questions are left for future study. It would be interesting to consider wider class of problems for example allowing for regularization terms that are convex but not strongly convex, and possibly non convex data fidelity terms. From an algorithmic point of view, it would be interesting to consider alternative approaches, for example considering stochastic methods. Finally, it would be interesting to investigate the numerical properties of the proposed method for practical problems. 

\appendix
\section{Auxiliary results}
We gather some relevant  results.
\subsection{Properties of the  dual diagonal function}

We first consider $\Rr_A $, $\ell^*_y$ defined in \eqref{eq:composite} and on the diagonal dual function $d_{\lambda}$ and its limit $d_0$  defined in \ref{eq:DualReg} and \ref{eq:Dual}, respectively. For similar results see also \cite{Garrigos2017}.

\begin{lemma}\label{L:functional properties}
Under the assumptions \ref{assumption:data-fit convex and coercive}-\ref{assumption:data-fit conditionned locally} and \ref{H:R strongly convex}-\ref{assumption:AR2},  we have:
\begin{enumerate}[label=(\roman*)]
	\item\label{L:functional properties:R_A Lipschitz} $\Rr_A$ is differentiable and $\nabla  \RA$ is Lipschitz continuous, with Lipschitz constant equal to $\sigma^{-1} \Vert A \Vert^2$.
	\item\label{L:functional properties:L at 0} $\forall y \in \Yy$, $\ell^*_{y}(0)=0$ and $\partial \ell^*_y(0)=\{y\}$.
	\item\label{L:functional properties:argmin d_0 nonempty} There holds: $\argmin  d_0 \neq \emptyset$.
 	\item\label{L:functional properties:decreasing function sequence} $\forall u \in \Yy$, the function $\lambda\in [0,+\infty) \mapsto \ d_{\lambda}(u)$ is nondecreasing. 
	\item\label{L:functional properties:from dual value to primal iterates} $\forall t >0, \forall u \in \Yy, \ \text{ if }
{x}:=\nabla R^*(-A^*u),$ then $
 \frac{\sigma}{2} \Vert {x} - x^\dagger \Vert^2 \leq d_0(u) - \inf d_0.$ 
 	\item\label{L:functional properties:diagonal rates via conditioning} $\forall u^\dagger \in \argmin  d_0$, if  $\lambda \Vert u^\dagger \Vert \leq \frac{\gamma}{q} \varrho^{q-1}$, then
 	\begin{equation}
 	    d_{\lambda}(u^\dagger) - \inf  d_0 \leq
 	    \begin{cases}
 	     0 & \text{ if } q=1, \\
 	     (1-\frac{1}{q}) {\gamma^{-1/(q-1)}} \Vert u^\dagger \Vert^{q/(q-1)} \lambda^{1/(q-1)} & \text{ if } q>1.
 	    \end{cases}
 	\end{equation}
\end{enumerate}
\end{lemma}
\begin{proof}
\ref{L:functional properties:R_A Lipschitz}: follows  from the strong convexity of $R$, see, e.g.,  
\cite[Theorem 18.15]{BauschkeCombettes2017}.\\
\ref{L:functional properties:L at 0}: it is a simple consequence of the properties of the Fenchel transform as it can be found, e.g., in 
 \cite[Proposition 13.10(i) \& Corollary 16.30]{BauschkeCombettes2017}.\\
\ref{L:functional properties:argmin d_0 nonempty} and \ref{L:functional properties:from dual value to primal iterates}: follow from \cite[Lemma 5]{Garrigos2017} by simply taking $f=R$ and $g=\delta_{\{\bar y\}}$, while property \ref{L:functional properties:decreasing function sequence} has been proved in \cite[Proposition 2(i)]{Garrigos2017}.\\
\ref{L:functional properties:diagonal rates via conditioning}: it is enough to verify that $\ell_{\bar y}(\cdot)$ is $q$-well conditioned in the sense of \cite[Definition 1]{Garrigos2017}, while assumption \ref{assumption:data-fit conditionned locally} holds only locally.
To check this, we introduce the function $\psi : \R \rightarrow \R$ defined for the $\varrho>0$ appearing in \ref{assumption:data-fit conditionned locally}  by:
\begin{equation}
 \psi  t \mapsto
\begin{cases}
 \frac{\gamma}{q}\vert t \vert^q  & \text{ if } \vert t \vert \leq \varrho, \\
 \frac{\gamma}{q} \varrho^{q-1} \vert t \vert & \text{ if } \vert t \vert > \varrho.
\end{cases}
\end{equation}
From \ref{assumption:data-fit conditionned locally}, we easily  deduce that $\ell_{\bar y}(y) \geq \psi(\Vert y- \bar y \Vert)$ for all $y \in \Yy$ (see \cite[Corollary 3.4.2]{ZalinescuBook2002}).
Note that $\psi$ is not convex for $q >1$, so in this case we consider instead the function
\begin{equation}
 m: \R \rightarrow \R , \quad t \mapsto 
\begin{cases}
 \frac{\gamma}{q}\vert t \vert^q  & \text{ if } \vert t \vert \leq q^{1/(1-q)} \varrho, \\
 \frac{\gamma}{q}\varrho^{q-1}\vert t \vert 
 -  \frac{\gamma}{q^{\frac{q}{q-1}}}\varrho^q (1 - \frac{1}{q})  & \text{ if } \vert t \vert > q^{1/(1-q)} \varrho,
\end{cases}
\end{equation}
and define $m:=\psi$ for $q=1$.
It is an easy exercise to verify that $m$ is indeed a convex function on $\R$, and that $m(w) \leq \psi(w)$ for all $w\in\R$.
Now, we can make use of \cite[Lemma 2]{Garrigos2017},
which tells us that $d_\lambda(u) - \inf d_0 \leq \lambda^{-1} m^*(\Vert u \Vert \lambda)$.
The desired result  now follows from the computation of the Fenchel transform of $m$.
If $q=1$, we have that $m(t) = \gamma \vert t \vert$, so classic Fenchel calculus entails that $m^*$ is $\delta_{[-\gamma,\gamma]}$, the indicator function of $[-\gamma,\gamma]$.
If $q>1$, easy computations show that $m^*$ reads
\begin{equation}
 m^* : \R \rightarrow \R, \quad s \mapsto
\begin{cases}
 (1 - \frac{1}{q}) \gamma^{-1/(q-1)}\vert s \vert^{\frac{q}{q-1}}  & \text{ if } \vert s \vert \leq \frac{\gamma}{q} \varrho^{q-1}, \\
 + \infty  & \text{ if } \vert s \vert > \frac{\gamma}{q} \varrho^{q-1}.
\end{cases}
\end{equation}
By now applying \cite[Lemma 2]{Garrigos2017} we  conclude. 
\end{proof}

\subsection{Useful tools for KL computations} 

In this section, we report some computations and properties concerning the KL divergence defined in \eqref{definition:KL}. 
For any $(u,y)\in (\R^d)^2$ we define $\mathrm{KL}(y,u)$ as in \eqref{definition:KL}.
Consider now the functions $\mathrm{KL}$ and $\mathrm{kl}$ with respect to the first argument only, and define $\mathrm{KL}_y(u):=\mathrm{KL}(y;u)$ and, similarly, its $i$-th component $\mathrm{kl}_{y_i}(u_i)$ for a fixed $y\in \R^d$. The component-wise expression for $\mathrm{KL}_y^*(w) = \sum_{i=1}^d \mathrm{kl}_{y_i}^*(w_i)$ can be then simply found by Fenchel calculus. It reads:
\begin{equation}  \label{definition:KL*}
\mathrm{kl}_{y_i}^*(w_i)=\begin{cases}
 -y_i\log(1-w_i)\quad&\quad\text{if }1-w_i>0\\
 +\infty\quad & \quad\text{otherwise}.
 \end{cases}
\end{equation}
\paragraph{Proximal maps}
For every $i=1,\ldots,d$, straightforward calculations show that
\begin{equation}   \label{eq:proxKL}
\text{prox}_{\frac{\tau}{\lambda}\mathrm{kl}_{y_i}}(u_i) = \frac{1}{2}\Big(u_i-\frac{\tau}{\lambda} + \sqrt{\left(u_i-\frac{\tau}{\lambda}\right)^2+4\frac{\tau}{\lambda}y_i}\Big).
\end{equation}
Furthermore, by applying Moreau's identity we have:
\begin{equation}  \label{eq:proxKL*}
\text{prox}_{\frac{\tau}{\lambda}\mathrm{kl}^*_{y_i}(\lambda\cdot)}(w_i)=\frac{1}{2\lambda}\Big( (1+\lambda w_i) - \sqrt{\left(1-\lambda w_i \right)^2 + 4\lambda\tau y_i} \Big).
\end{equation}

The following lemma implies the $q$-conditioning of the Kullback-Leibler divergence.

\begin{lemma}[$2$-conditioning of the KL data-fit term]\label{P:conditioning of KL}Let $\bar y \in \R^d$ and $\varrho \in (0,+\infty)$.
Then,
\begin{equation}
    (\forall y \in \mathbb{B}(\bar y, \varrho)) \quad 
    \mathrm{KL}(\bar y , y) \geq 
    \left(\frac{1}{\varrho c^2} + \frac{1}{\varrho^2 c}\ln{\frac{c}{\varrho + c}} \right)
    \Vert y-\bar y \Vert^2,
    \text{ where } c = d \Vert  \bar{y} \Vert_\infty.
\end{equation}{}
\end{lemma}
\begin{proof}
Let $y \in \mathbb{B}(\bar y, \varrho)$.
By \cite[Lemma 10.2]{Garrigos2017}, we have that
\begin{equation}\label{cok1}
    \mathrm{KL}(\bar y , y) \geq c m(\Vert y-\bar y \Vert), \text{ where } m(t) = c^{-1}\vert t \vert - \ln \left(1 + c^{-1}\vert t \vert\right).
\end{equation}
To  get the desired result, we need to find a quadratic lower bound for $m$ over $[-\varrho,\varrho]$.
For simplicity, let us consider the change of variable $s = c^{-1}\vert t \vert \in [0,c^{-1}\varrho]$.
Since the statement is trivially valid for $y = \bar y$, we can assume that $s > 0$ and write
\[
    s-\ln (1+s) = s^2 \phi(s), \text{ where }  \phi(s) := \frac{s-\ln (1+s)}{s^2}.
\]
To conclude, we only need to verify that $\phi$ is decreasing on $]0,+\infty[$.
Indeed, this would imply that $m(t) \geq c^{-2}t^2 \phi(c^{-1}\varrho)$, which together with \eqref{cok1} would complete the proof.
To see that $\phi$ is decreasing, we compute explicitly its derivative on $]0, + \infty[$ and see that $\phi'(s) \leq 0$ if and only if $\psi(s):=s(s+2)-2(1+s) \ln(1+s) \geq 0$. Combining this with the fact that $\psi(0)=0$, and that $\psi'(s) = 2(s - \ln(1+s))$ is positive $]0, + \infty[$ we conclude the proof.
\end{proof}

In the following we deal with  proximal points of the dual of the KL divergence, corresponding to noise-free and noisy data $\bar y$ and $\hat y$, respectively. As shown in Proposition \ref{prop:KL_approx} a type 2 approximation in the sense of Definition \ref{def:approximation} holds. The following proposition is a one-dimensional counterexample showing that a type 3 approximation -- for which better convergence rates can be obtained -- does not hold.

\begin{proposition}\label{L:KL not type 3 error}
Let $w \in \R$ and  $\bar y,\hat y \in ]0,+\infty[$. If 
$ \mathrm{prox}_{\mathrm{kl}^*_{\hat{y}}}(w) ~\approx^{\varepsilon}_{\text{\tiny 3}} \mathrm{prox}_{\mathrm{kl}^*_{\bar{y}}}(w)$ holds in the sense of Definition \ref{def:approximation} for some $\varepsilon >0$, 
then
\[
\varepsilon \geq \frac{2\vert \hat y - \bar y \vert}{(1-w) + \sqrt{(1-w)^2 + 4\hat y}}.
\]
In particular, $\varepsilon \to + \infty$ when $w \to +\infty$.
\end{proposition}

\begin{proof}
Let $\varepsilon \geq 0$ such that the type 3 approximation property holds. By Definition \ref{def:approximation}, there exists $e \in \R$ such that $\vert e \vert \leq \varepsilon$ and 
$\mathrm{prox}_{\mathrm{kl}^*_{\hat{y}}}(w) = \mathrm{prox}_{\mathrm{kl}^*_{\bar{y}}}(w+e)$.
Using the formula \eqref{eq:proxKL*}, we see that this is equivalent to
\begin{equation}
    \frac{1}{2}\left[(1+w)-\sqrt{(1-w)^2+4\hat y}\right]
    =
    \frac{1}{2}\left[(1+w+e)-\sqrt{(1-w-e)^2+4\bar y}\right].
\end{equation}
and we complete the proof by noting that the above equality is equivalent to
\begin{equation}
    e
    \frac{1}{2}\left[(1-w)+\sqrt{(1-w)^2+4\hat y}\right]
    = \bar y - \hat y.
\end{equation}
\end{proof}

\subsection{Miscellaneous}
We here recall some technical lemmas which are used in several sections of the manuscript.
The following Lemma is useful to characterize the speed of decay of the diagonal term $\lambda(\cdot)$ in assumption \ref{requirement:lambda1:cont}, see also Remark \ref{remark:summability:cont}.

\begin{lemma}   \label{lemma:integrability}
Let $\lambda: \R_+\to \R_+$ a decreasing function such that $\int_{\R_+} |\lambda(t)|^{1/2}~dt < +\infty$. Then, the function $t\mapsto t\lambda(t)$ is integrable on $\R_+$.
\end{lemma}

\begin{proof}
We first show that the function $t\mapsto  t \sqrt{\lambda(t)}$ tends to zero as $t\to +\infty$. We have that for every $T>0$:
$$
\int_{T/2}^{+\infty} \sqrt{\lambda(t)}~dt \geq \int_{T/2}^{T} \sqrt{\lambda(t)}~dt \geq \frac{T}{2} \sqrt{\lambda(T)},
$$
where the last inequality follows from the decreasing property of $\lambda$ in the interval $[T/2,T]$. By taking limits, we get the required property:
$$
\lsup_{T\to +\infty}~ \frac{T}{2} \sqrt{\lambda(T)} \leq \lim_{T\to +\infty}~ \int_{T/2}^{+\infty} \sqrt{\lambda(t)}~dt = 0.
$$
Now, from the observation
$$
\lim_{t\to +\infty}~\frac{t\lambda(t)}{ \sqrt{\lambda(t)}} = \lim_{t\to +\infty}~t\sqrt{\lambda(t)}=0,
$$
we deduce that there exists some $\overline{T}>0$ such that $t\lambda(t)\leq \sqrt{\lambda(t)}$ for all $t\geq \overline{T}$. By thus taking $T>\overline{T}$, we have:
$$
\int_{0}^T t\lambda(t)~dt = \int_{0}^{\overline{T}} t\lambda(t)~dt +  \int_{\overline{T}}^T t\lambda(t)~dt \leq  \int_{0}^{\overline{T}} t\lambda(t)~dt +  \int_{\overline{T}}^T \sqrt{\lambda(t)}~dt, 
$$
which by taking the supremum over all $T>\bar{T}$ on both sides entails:
$$
\int_{\R_+}  t\lambda(t)~dt \leq  \int_{0}^{\overline{T}} t\lambda(t)~dt +  \int_{\overline{T}}^{+\infty} \sqrt{\lambda(t)}~dt < +\infty.
$$
\end{proof}

Next, we state and prove a variant of \cite[Lemma 5.14]{Attouch2015} which we have used in the proof of Theorem \ref{error:estimate:iterates} to get the final stability estimate \eqref{error:estimate:iterates}.

\begin{lemma}\label{lemma:attouch (sequences)}
Let $(a_k)_\kin$, $(b_k)_\kin$ and $(c_k)_\kin$ be positive sequences, and assume that $c_k$ is increasing.
If
\begin{equation}
    (\forall \kin) \quad a_k^2 \leq c_k + \sum\limits_{j=1}^{k-1} b_j a_{j+1},
\end{equation}
then $\max\limits_{j=1,...,k} a_j \leq \sqrt{c_k} + \sum\limits_{j=1}^{k-1} b_j$,  for every $k\in\mathbb{N}$.
\end{lemma}
\begin{proof}
Take $\kin$, and let $A_k:=\max\limits_{m=1,...,k} a_m$.
Then, for all $1 \leq m \leq k$:
\begin{equation}
    a_m^2 \leq  c_m + \sum\limits_{j=1}^{m-1} b_j a_{j+1}
    \leq c_k + A_k \sum\limits_{j=1}^{k-1} b_j,
\end{equation}
because $c_k$ is increasing and $b_j$ is positive.
Therefore $A_k^2 \leq c_k + A_k \sum\limits_{j=1}^{k-1} b_j$.
Define $S_k= \sum\limits_{j=1}^{k-1} b_j$. By computing and bounding the solutions of the previous inequality we conclude that
\[
    A_k \leq \frac{S_k + \sqrt{S_k + 4c_k} }{2} \leq S_k + \sqrt{c_k}.
\]

\end{proof}
We recall a useful characterisation of the elements in the $\varepsilon$-subdifferential of a function in $\Gamma_0(\mathcal{H})$. This property is used to  prove Proposition \ref{prop:KL_approx}, see also \cite{ZalinescuBook2002}.

\begin{lemma}[Theorem 2.4.2, \cite{ZalinescuBook2002}]  \label{lemma:characterisation_eps}
Let $\mathcal{H}$ be an Hilbert space, let $f\in\Gamma_0(\mathcal{H})$, let $(x,u)\in\mathcal{H}^2$, and let $\varepsilon>0$. 
Then, the following statements are equivalent:
\begin{enumerate}
\item[i)] $u\in\partial_\varepsilon f(x)$;
\item[ii)] The following $\varepsilon$-Young-Fenchel inequality holds: 
\begin{equation}  
f(x) + f^*(u) \leq \langle u, x \rangle +\varepsilon;
\end{equation}
\item[iii)] $x\in\partial_\varepsilon f^*(u)$.
\end{enumerate}
\end{lemma}

\bibliographystyle{plain}
\bibliography{biblioAcc3D}
\end{document}